\documentclass[11pt,fleqn]{amsart} 

\usepackage[usenames,dvipsnames,condensed]{xcolor}

\usepackage{amsthm}
\usepackage{amsfonts}
\usepackage[english]{babel}
\usepackage[usenames]{xcolor}
\usepackage{graphicx}
\usepackage{soul}
\usepackage{stfloats}
\usepackage{morefloats}
\usepackage{cite}
\usepackage{lscape}
\usepackage{epstopdf}
\usepackage{braket}
\usepackage[lite]{amsrefs}
\usepackage{mathbbol}
\usepackage{tikz}

\usepackage{amsmath,amstext,amsopn,amsfonts,eucal,amssymb}
\usepackage{graphicx,wrapfig,url}

\newcommand\N{{\mathbb N}}
\newcommand\Z{{\mathbb Z}}

\newtheorem{theorem}{Theorem}[section]
\newtheorem{corollary}[theorem]{Corollary}
\newtheorem{lemma}[theorem]{Lemma}
\newtheorem{proposition}[theorem]{Proposition}
\newtheorem{definition}[theorem]{Definition}

\newtheorem{example}[theorem]{Example}

\newtheorem{remark}[theorem]{Remark}

\newcommand{\cal}{\mathcal}

\begin{document}

\title{Fundamental Heaps for  Surface Ribbons  and  Cocycle Invariants}

\author{Masahico Saito} 
\address{Department of Mathematics, 
	University of South Florida, Tampa, FL 33620, U.S.A.} 
\email{saito@usf.edu} 

\author{Emanuele Zappala} 
\address{Institute of Mathematics and Statistics, University of Tartu\\
	Narva mnt 18, 51009 Tartu, Estonia} 
\email{emanuele.amedeo.zappala@ut.ee \\ zae@usf.edu}

\maketitle

\begin{abstract}
We introduce the notion of fundamental heap for compact orientable surfaces with boundary embedded in $3$-space, which is an isotopy invariant of the embedding.  It is a group, endowed with a ternary heap operation, defined using diagrams of surfaces in a form of thickened trivalent graphs
called surface ribbons.
We prove that the fundamental heap has a free part whose rank is given by the number of connected components of the surface. We study the behavior of the invariant under boundary connected sum, as well as addition/deletion of twisted bands, and provide formulas relating the number of generators of the fundamental heap to the Euler characteristics. We describe in detail the effect of stabilization on the fundamental heap, and determine that for each given finitely presented group 
there exists a surface ribbon 
whose fundamental heap is isomorphic to it, up to extra free factors.
A relation between the fundamental heap and the Wirtinger presentation is also described. Moreover, we introduce cocycle invariants for surface ribbons  
using the notion of mutually distributive cohomology and heap colorings. 
Explicit computations of fundamental heap and cocycle invariants are presented.
\end{abstract}

\date{\empty}

\tableofcontents

\section{Introduction}

The purpose of this article is to introduce, and investigate, invariants of compact orientable surfaces with boundary  embedded in $3$-space up to isotopy, using 
ternary self-distributive (TSD) operations and their cohomology theory. More specifically, we focus on {\it heaps}, ternary structures that are epitomized by the operation $(x,y,z)\mapsto xy^{-1}z$ in a group $G$, and the notion of {\it mutually distributive} cocycles of TSD operations, applied to heaps. 
Any compact orientable surface embedded in 3-space can be represented by a thin ribbon neighborhood of a trivalent graph that is its spine, which we call a {\it surface ribbon}. Their diagrams and
Reidemeister type  moves were studied in \cite{Matsu}. We utilize this diagrammatic.

Self-distributive (binary) operations have been used since the 1980's to construct invariants of knots and links, following the articles \cite{Joyce,Mat}, where the notion of {\it fundamental quandle} was introduced,
defined topologically and diagrammatically. Homology and cohomology theories of quandles were then introduced, and used to construct 
 invariants of links in $3$-space, as well as knotted surfaces in $4$-space \cite{CJKLS}. These invariants are defined via certain partition functions, roughly described in the case of links in $3$-space as follows. The initial data of the construction is a quandle $X$, along with a $2$-cocycle of $X$ with coefficients in an abelian group $A$. First, one defines the set of $X$-colorings of a fixed diagram $D$ of a link $L$ as the set of homomorphisms from the fundamental quandle of $L$ (obtained through $D$) to $X$. 
 A coloring is also regarded as an assignment of elements of $X$ to arcs of $D$, and assigned elements are called {\it colors}.
  For each coloring, then, one takes the Boltzmann weights of each crossing of $D$, where the $2$-cocycle is evaluated at the pair of colors of the underpassing and overpassing arcs, then 
for each coloring,  
all these weights are multiplied together 
over all crossings. Upon summing over all $X$-colorings  
this quantity results to be invariant with respect to Reidemeister moves and, therefore, is independent of the choice of diagram of $D$.

\begin{figure}[hbt]
	\begin{center}
		\begin{tikzpicture}
		\draw[->] (2,2) -- (0,0);
		\draw (0,2)--(0.9,1.1);
		\draw[->] (1.1,0.9) -- (2,0);
		
		\draw[->] (5,2)--(7,0); 
		\draw (7,2)--(6.1,1.1);
		\draw[->] (5.9,0.9)--(5,0);
		
		\node at (0,1.7) {$x$};
		\node  at (2.1,0.45) {$x*y$};
		\node at (2,1.7) {$y$};
		
		\node  at (4.9,0.45) {$x$};
		\node  at  (5,1.7) {$y$};
		\node at (7.1,1.55) {$x*y$};
		\end{tikzpicture}
	\end{center}
	\caption{Positive (left) and negative (right) crossings and their colorings for binary quandles.}
	\label{fig:positivenegative}
\end{figure}
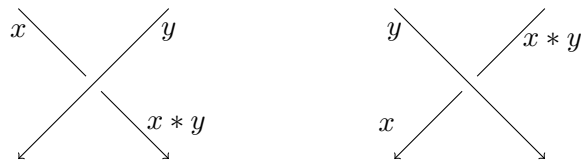

The intuitive and diagrammatic interpretation that underlies the 
above described 
paradigm relies on the scheme depicted in Figure~\ref{fig:positivenegative}, where the orientation is used to determine the sign of the crossing and the consequent sign of the Boltzmann weight. 
In the figure, a coloring rule is given at a crossing, where two colors $x, y$ determine
the third color to be $x*y$, where $*$ denotes the quandle operation. 
Using TSD operations, a similar diagrammatic interpretation is introduced, where the arcs are doubled and the colors change at given crossings by means of the TSD operation, following the rules depicted in Figure~\ref{buildingblocks} (A), where $z=x u^{-1} v$ and $w=y u^{-1} v$. 
The colors of the underpassing arcs change from $x$, resp. $y$, to $z=T(x,u,v)$, resp. $w=T(y,u,v)$, where $T$ is a given TSD operation. Among known examples of TSD operations we have compositions of binary self-distributive operations 
of  
mutually (binary) self-distributive operations \cite{ESZcascade}, as well as heaps, which are 
not compositions of lower arity operations. The ribbon cocycle invariant 
is a framed link invariant constructed in \cite{SZframedlinks}, 
 using the diagrammatic interpretation of heaps 
 given above. Concisely, in order to define the ribbon cocycle invariant the notion of fundamental heap is introduced, and consequently, that of heap coloring of a framed diagram of a framed link as well. For each heap coloring with a given heap $X$ the Boltzmann 
 weight is assigned
 at each crossing from 
 two terms associated at each instance of $T$ as described above, and the weights derived from evaluating 
 a fixed TSD $2$-cocycle are then multiplied
 over all crossings. 
  Summing over all colorings by $X$ of the given framed link diagram one obtains an object that is invariant under framed Reidemeister moves. Interestingly, the fact that the Boltzmann weights are defined independently on each of the two 
  underpassing arcs  
  at a crossing, by means of the ternary operation, induces an invariant that is an element of the tensor product of algebras $\mathbb Z[A]\otimes \mathbb Z[A]$, where $A$ is the abelian coefficient ring used for cohomology. This is a fundamental difference between the binary and the ternary approaches.

\begin{figure}[htb]
	\begin{center}
		\includegraphics[width=4in]{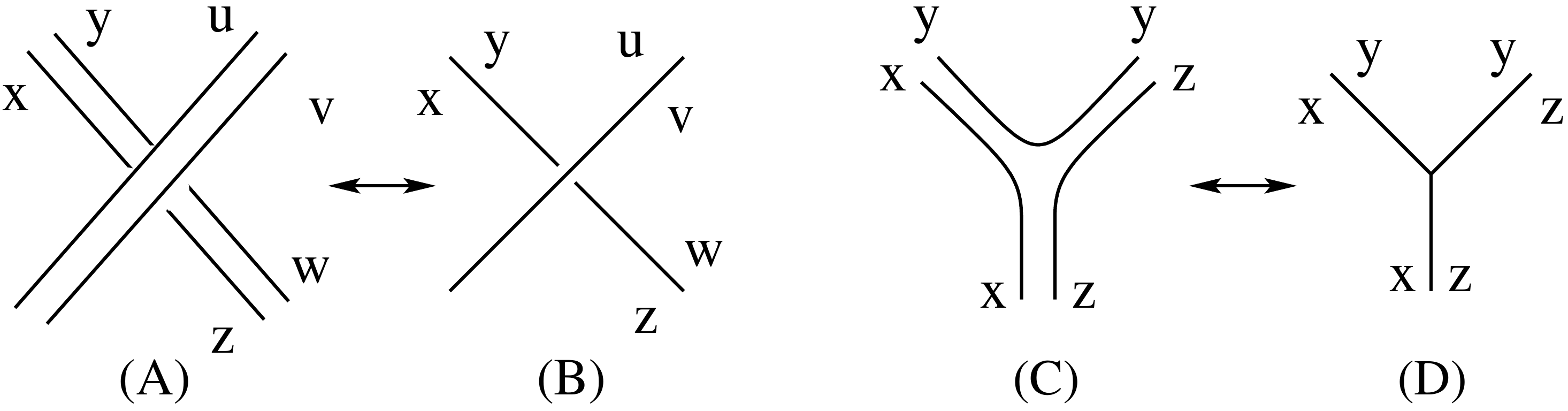}
	\end{center}
	\caption{Building blocks}
	\label{buildingblocks}
\end{figure}

In the present article, we employ the same principles used to obtain the fundamental heap and the ribbon cocycle invariant for framed links \cite{SZframedlinks} to 
define those for 
surface  ribbons.
Differences from earlier work \cite{SZframedlinks} are described as follows.
These objects are 
proved to be isotopy invariants 
using 
 local moves that consist of framed Reidemeister moves, as well as moves that involve trivalent fat vertices
 of surface ribbons~\cite{Matsu}, 
 which present the main difference with respect to the case of framed links. In addition, since the boundary components of surface ribbons need to be oriented in antiparallel fashion, a further difference appears between the cocycle invariants of framed links and those of surface ribbons. Lastly, whilst the ribbon cocycle invariant 
 for framed links 
 employs by definition a single heap $2$-cocycle with coefficients in some abelian group $A$, 
  the initial data to construct cocycle invariants of surface ribbons is a family of mutually distributive $2$-cocycles, in the sense of \cite{ESZcascade},
 assigned to each connected component of a surface ribbon. 
 Then, as in the case of framed links, the invariant takes values in  a tensor product of copies of the group algebra of $A$, where tensor factors correspond to boundary 
connected components. Thus the invariant is stratified to connected components of both surface ribbons and their boundary curves.

Related topics can be found, for example,  in the following papers.
Spatial graphs with a move that corresponds to handle slides have been studied 
also for handlebody-links \cite{Ishii}. Corresponding algebraic structures that have multiplication and braiding at the same time, with compatibility conditions, have also been studied~\cite{CIST,Lebed}. 
Invariants for compact surfaces with boundary  represented by ribbon graphs
using the moves provided in~\cite{Matsu} were defined and studied in~\cite{IMM,SZbraidedFrob}.
 Knot invariants using ternary operations have been studied, for example, in 
	  \cite{NOO,Nie1,Nie2}, in which colorings are assigned to complementary regions,
	  while in this paper, colorings are assigned on doubled arcs of 
	  surface ribbons.

\begin{figure}[htb]
\begin{center}
\includegraphics[width=3.5in]{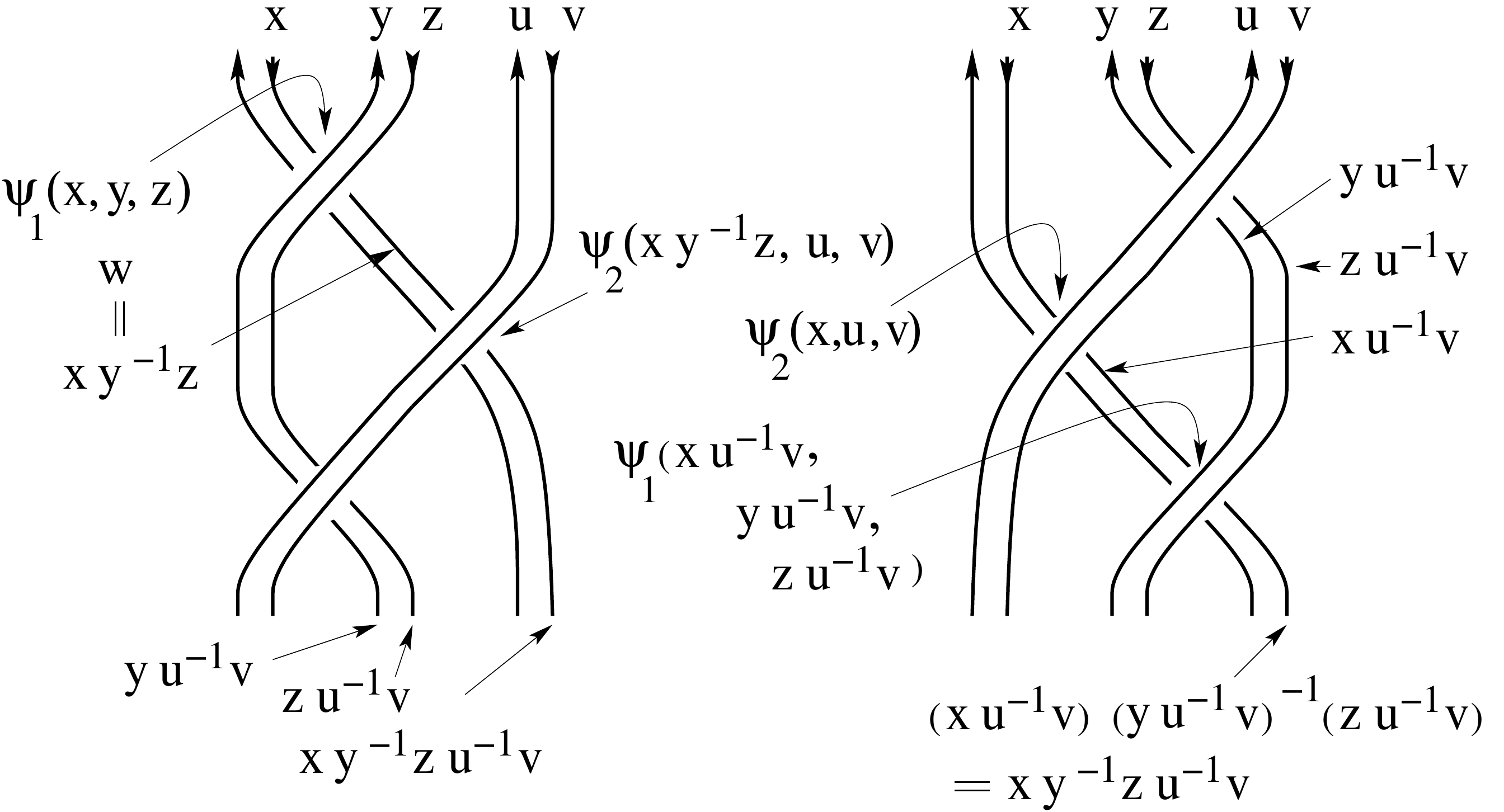}
\end{center}
\caption{Heap and Reidemeister type III move}
\label{heaptypeIIIcocy}
\end{figure}

 We now present more details regarding the constructions of invariants.
The fundamental heap of surface ribbons is defined from a given diagram $D$ of $S$ 
 as follows. We first introduce a generator for each arc appearing in $D$. 
Next, using the coloring condition as explained above,  one introduces relations.
The fundamental heap is the group defined by a presentation with these generators and relations.
 Its isomorphism class
 is invariant under 
the moves of Figure~\ref{moves} given by \cite{Matsu}, and it is therefore independent 
of 
the choice of $D$, and is an 
 isotopy invariant.  
We present several properties of the fundamental heap. 
For instance, we prove that it is not changed by addition/deletion of twisted bands, with the application of implying certain inequalities between the genus of the ribbon surface and the number of generators 
of the fundamental heap. 
Moreover, we prove that any surface ribbon can be turned, by means of stabilizations, into a new surface ribbon whose fundamental heap is free. 
 We also find a solution to the realization problem for heaps as fundamental heaps of surface ribbons. Specifically, given a finitely presented 
group 
$X$, there exists a surface ribbon whose fundamental heap 
is isomorphic to  the group heap of
$X$, up to some free factor. 
The rank of the free factor is related to the Euler characteristic and the number of connected components.

The framed Reidemeister move III with the antiparallel convention for boundary components, and heap colorings, is given in Figure~\ref{heaptypeIIIcocy}. 
The figure also indicates
 that different ribbons (belonging to different connected components) are decorated with possibly different $2$-cocycles, with the fundamental assumptions that all pairs of cocycles are mutually distributive. Then, the coloring condition at a framed Reidemeister move III is guaranteed by the self-distributivity of heap operations, while the invariance of weights is equivalent to mutual distributivity of the cocycles. Moreover, the presence of  trivalent fat vertices also requires extra conditions on the {\it labeled cohomology} of \cite{ESZcascade}. These conditions, hereby called {\it reversibility} and {\it additivity}, ensure that ribbons can be slid above and below fat vertices, which is one of the moves 
that determine the isotopy class of the embedding 
 \cite{Matsu}. 

The article is organized as follows. In Section~\ref{sec:Pre} we recall some algebraic and topological preliminaries. More specifically, we provide the definition of surface ribbons and their diagrams, and we recall the diagrammatic moves \cite{Matsu} that determine their isotopy class. Then we introduce heaps, and recall ternary self-distributive (co)homology and labeled cohomology in the specific case of mutually distributive $2$-cocycles. In Section~\ref{sec:Fundheap} we introduce the fundamental heap of surface ribbons. Then we study the properties of fundamental heaps under stabilization, twisted band addition/deletion, and boundary connected sum.
 We investigate its relation to Euler characteristic and genus of surface ribbons, and provide a positive answer to the realization problem of heaps as fundamental heaps of surface ribbons. A connection with the Wirtinger presentation is also provided, as well as  
some 
classes of examples. Section~\ref{sec:RAconditions} is devoted to the definition of a subgroup of the mutually distributive second cohomology group of heaps, determined by two additional conditions. Families of examples of cocycles satisfying the extra conditions are also provided. In Section~\ref{sec:cocyinvariant} we introduce colorings of  surface ribbons by heaps, and use this notion along with the cohomology of Section~\ref{sec:RAconditions} to construct cocycle invariants of  surface ribbons. We provide nontrivial examples and 
discuss
 a formula for the cocycle invariants of boundary connected sums.

\section{ Preliminaries }\label{sec:Pre}

In this section we review materials used in this paper.

\subsection{Diagrams of surface ribbons  and their moves}

In this section we review diagrams representing compact 
orientable 
surfaces with boundary embedded in 3-space (spatial surfaces with boundary). Our discussion is based on \cite{Matsu}. By compact surfaces with boundary, we mean surfaces that are compact and such that each component has a non-empy boundary. Compact surfaces with boundaries embedded in $3$-space are determined by their {\it spines}. Recall that a spine for a surface $S$ is a trivalent graph $G$ such that a normal neighborhood  of $G$
in $S$ 
 is homeomorphic to $S$ with a normal neighborhood of $\partial S$ removed. We therefore represent compact surfaces with boundary, diagrammatically, as  fattened 
 trivalent graphs where each edge is given by a pair of parallel arcs, while vertices are represented by triples of arcs as in Figure~\ref{buildingblocks} (C).
     We call such representations {\it surface ribbons} throughout the paper. 
     Thus a surface ribbon is a compact orientable surface with boundary in the form of a thickened  flat trivalent graph. 
  The fundamental diagrammatic units are given in Figure~\ref{buildingblocks}, where in (A) a fattened crossing is represented. 
  For simplicity we also represent surface ribbons by trivalent graphs as in (B) and (D) in the figure.

\begin{figure}[htb]
	\begin{center}
		\includegraphics[width=3in]{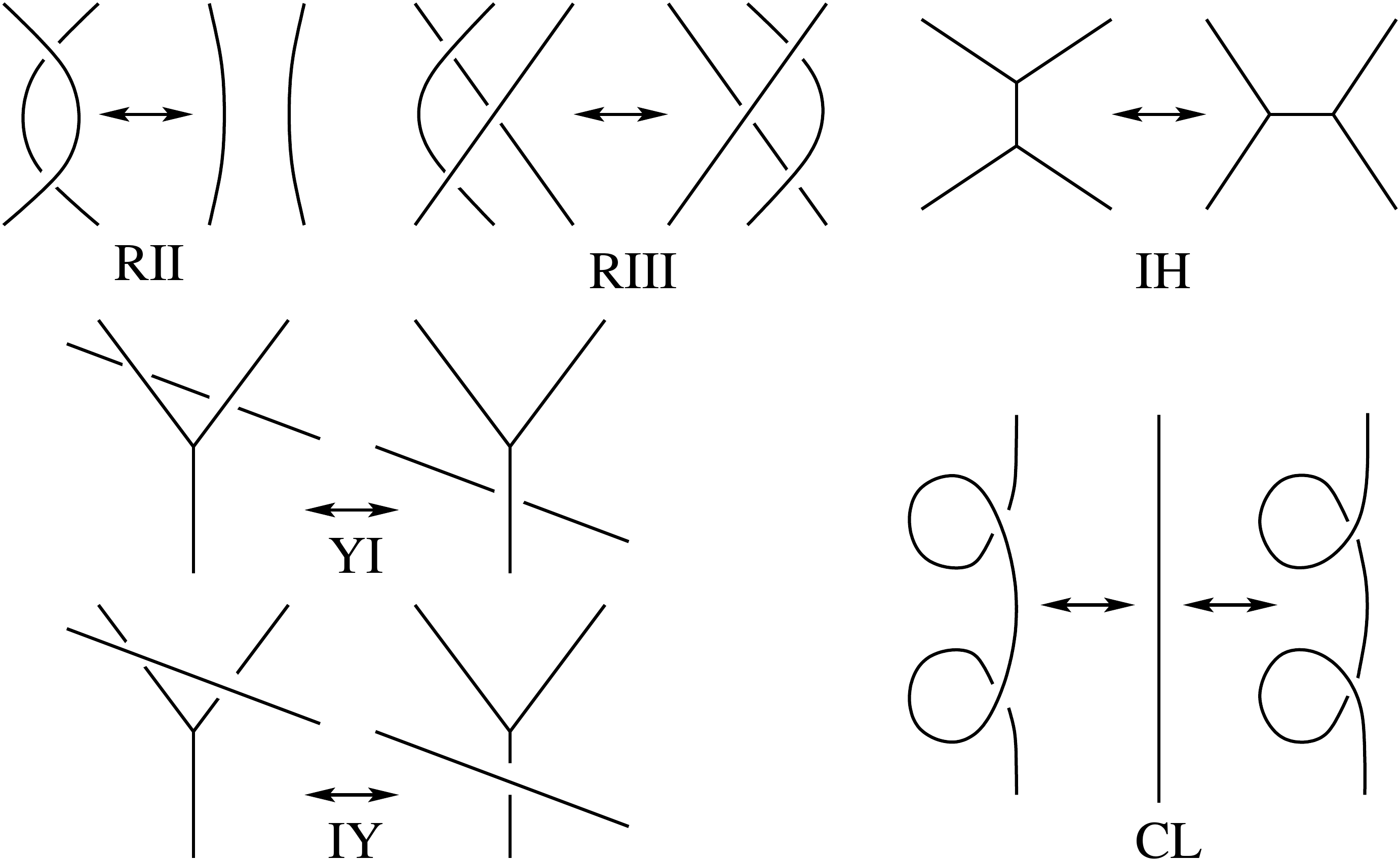}
	\end{center}
	\caption{Moves}
	\label{moves}
\end{figure}

In \cite{Matsu}, it was shown that the isotopy class of a compact orientable surface with boundary
in a surface ribbon form 
 is determined diagrammatically by the moves given in Figure~\ref{moves}. Moves RII, RIII and CL are the framed Reidemeister moves for framed links. Moves IY, YI and IH appear also in the study of handlebody knots in $3$-space, see for instance \cite{Ishii}. In particular, we mention that the IH move is important in the well posedness of the diagrammatic interpretation in terms of trivalent graphs (spines), since it allows to arbitrarily desingularize higher order vertices. 
%
Matsuzaki has determined the moves for non-oriented surfaces as well \cite{Matsu}, although we do not consider this case here. 
The main difference with the present case is that 
 a half-twist 
 is specified in trivalent graphs, and 
  further moves involving half-twists 
  are considered as well. 
  In the context of orientable surfaces no half-twist needs to be taken into account, as 
  half twists appear in even numbers for orientable surfaces, and 
  two half-twists 
  are represented by a small loop as in Figure~\ref{loop}.

\begin{figure}[htb]
\begin{center}
\includegraphics[width=1.2in]{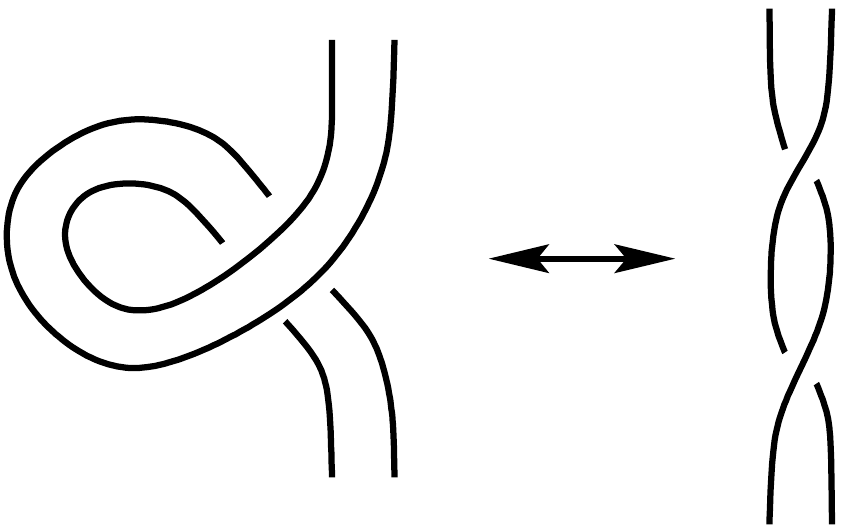}
\end{center}
\caption{A loop corresponds to a full twist}
\label{loop}
\end{figure}

\subsection{Heaps}

In this section we recall the definition and basic properties of heaps.
Given a set $X$ with a ternary operation $[ - ]$, the set of equalities 
$$ 
 [  [ x_1, x_2, x_3 ], x_4, x_5  ]=  [  x_1,[  x_4, x_3, x_2 ]  ,  x_5 ] = [  x_1,x_2, [  x_3, x_4, x_5 ]  ]   
$$ 
 is called {\it para-associativity}.
The 
equations $[x,x,y]=y$ and $[x, y, y ] = x$ are called the {\it degeneracy} conditions.
A {\it heap} is a non-empty set with a ternary operation satisfying 
		the para-associativity  
		and the degeneracy conditions~\cite{ESZheap}.

A typical example of a heap is a group $G$ where the ternary operation is given by $[x,y,z]=xy^{-1}z$,
which we call a {\it group heap}. If $G$ is abelian, we call it an {\it abelian (group) heap}. 
Conversely, 
given a heap $X$ with a fixed element $e$,  
one defines a binary operation on $X$ by $x*y=[x,e,y]$ which makes $(X,*)$ into a group with $e$ as the identity,  and the inverse of $x$ is $[e,x,e]$ for any $x \in X$.  Moreover, the associated group heap coincides with the initial heap structure. Focusing on group heaps is therefore not a strong restriction, as it is always possible to construct a group whose group heap coincides with an arbitrary heap.

Let $X$ be a set   with a ternary operation  $(x,y,z)\mapsto  T(x,y,z)$.
The condition 
$ T ( ( x, y,z ) ,  u,v)= T ( T(x,u,v), T(y,u,v) T(z,u,v ) )$ for all $x,y,z,u,v \in X$, is called {\it  ternary 
	 self-distributivity}, TSD for short. 
It is known and easily checked that the heap operation $(x,y,z)\mapsto  [x,y,z]=T(x,y,z)$ is ternary self-distributive.
In this paper we focus on the TSD structures of  group heaps.

\subsection{Ternary self-distributive homology}\label{sec:TSDcoh}

The ternary self-distributive (co)homology, which we review, was studied in \cite{EGM,Green,ESZcascade}.
Let $X$ be a ternary self-distributive set. The $n$-dimensional chain group $C_n^{\rm SD}(X)$ 
is the free abelian group generated by 
$(2n-1)$-tuples $(x_1, x_2, \ldots, x_{2n-1})$.
The boundary operator $d_n: C_n^{\rm SD} (X) \rightarrow  C_{n-1}^{\rm SD} (X)$ 
is defined by 
\begin{eqnarray*}
\lefteqn{ d_n (x_1, x_2, \ldots, x_{2n-1}) =}\\
& &  \sum_{i=1}^{n} (-1)^i [\ (x_1, \ldots, \widehat{x_{2i},x_{2i+1}},\ldots, x_{2n-1}) \\
& & \hspace{15mm}
- (x_1 x_{2i}^{-1}x_{2i+1},  \ldots, x_{2i-1} x_{2i}^{-1} x_{2i+1} , \widehat{x_{2i},x_{2i+1}}, x_{2i+2}, \ldots  , x_{2n-1}) \ ] .
\end{eqnarray*}
Cycle, boundary, homology groups are as usual denoted by $Z_n^{\rm SD} (X)$,
$B_n^{\rm SD} (X)$, and  $H_n^{\rm SD} (X)$, respectively. For an abelian group $A$, one defines the cochain, cocycle, coboundary and cohomology groups by dualizing their homological counterparts, as usual. A similar notation, with upper indices, is used to indicate these groups.  We adopt the convention that the cohomology differentials, written as $\delta^n : C^n_{\rm SD}(X,A) \longrightarrow C^{n+1}_{\rm SD}(X,A)$, are dual to the homological differentials $d_{n+1}: C_{n+1}^{\rm SD}(X) \longrightarrow C_n^{\rm SD}(X)$. 

The 2-cocycle condition in this  cohomology is formulated as 
 $$
(*) \quad 
\delta^2 \psi(x,y,z,u,v) = \psi(x,y,z)  - \psi(xu^{-1}v,yu^{-1}v,zu^{-1}v) - \psi(x,u,v)  + \psi(xy^{-1}z,u,v) = 0,
$$
where $x,y,z,u,v \in X$.

In addition, we will need the notion of 
mutually distributive cocycles \cite{ESZcascade}.
Although 
its definition was given
 in \cite{ESZcascade} for more general settings, 
 here we provide the definition for the special case 
 that we apply in this paper.  
Let $(X,T)$ denote a TSD structure and let $\psi_1$ and $\psi_2$ be TSD $2$-cocycles (as given above) with 
an abelian coefficient group $A$. Then, we say that the pair $(\psi_1,\psi_2)$ is 
{\it mutually distributive} 
 if the following two conditions hold
\begin{eqnarray*}
\psi_1(x,y,z) + \psi_2(T(x,y,z),u,v) &=& \psi_2(x,u,v) + \psi_1(T(x,u,v),T(y,u,v),T(z,u,v)) , \label{eq:mutdist1}\\
\psi_2(x,y,z) + \psi_1(T(x,y,z),u,v) &=& \psi_1(x,u,v) + \psi_2(T(x,u,v),T(y,u,v),T(z,u,v)). \label{eq:mutdist2}
\end{eqnarray*}
 In this situation we also say that $\psi_1$ and $\psi_2$ are mutually distributive.

A pair of mutually distributive 2-cocycles $(\psi_1, \psi_2)$ is called {\it coboundary} if 
there exists $f \in C^1_{\rm SD} (X,A)$ such that $\psi_i=\delta f $ for $i=1,2$.

\section{The fundamental heap of surface ribbons }\label{sec:Fundheap} 

In this section we define the fundamental heap, we show that it is an invariant of  surface ribbons, and present examples and properties. 

\subsection{Definitions and examples}

\begin{definition}\label{def:fund}
{\rm
The {\it fundamental heap} $h(S)$ of 
 a surface ribbon $S$ is defined as follows.
Let $D$ be a diagram of $S$ with double arcs of ribbons with building blocks as in Figure~\ref{buildingblocks} (A) at crossings and (C) at trivalent vertices.
We define $h(D)$ by a presentation using $D$ and show that it is well defined,  i.e. independent of choice of $D$.
	Let ${\cal  A}$ be the set of   arcs.
	Two arcs of a ribbon segment (doubled arcs) are listed as separate (distinct) elements of ${\cal  A}$. 
Each arc is assigned a generator. In 
 Figure~\ref{buildingblocks}, 
 generators are represented by letters 
(labels) $x,y,u,v,z,w$. Letters  assigned to arcs are identified with (the names of)  the arcs themselves, and 
regarded as elements of ${\cal A}$. 
Then the set of generators of $h(D)$ is ${\cal A}$.

For each crossing as depicted in  Figure~\ref{buildingblocks} (A),  
the relations are  given by $\{ z=x u^{-1} v, w=y u^{-1} v \}$. 
Specifically, when the arc $x$ goes under the arcs $(u, v)$, in this order, to the arc $z$, then 
the relation is defined as $ z=x u^{-1} v$, and similar from $y$ to $w$.
The set of union of the two relations over all crossings is denoted by ${\cal T}$ and constitutes 
the set of relations of $h(D)$. 
For each trivalent vertex as in Figure~\ref{buildingblocks} (C), each connected arc receives the same letter, and no relation is imposed. 

The fundamental heap $h(D)$ is the group heap of the group whose presentation
 is given by a set of generators corresponding to double arcs, and the set of relations assigned to all crossings:
$
\langle
\,
{\cal A}  \ | \ 
{\cal T} 
\,
\rangle 
$. 
In the next lemma, it is proved that $h(D)$ does not depend on the choice of $D$ and, therefore that it is well defined for $S$, and it is denoted by $h(S)$.
For a connected disk $B^2$, it is defined as $h(B^2)=\Z$. 
}
\end{definition}

For diagrams of spine consisting of single arcs as in Figure~\ref{buildingblocks} (B) and (D), 
the letters (generators) assigned are placed at the two sides of each arc.

\begin{figure}[htb]
\begin{center}
\includegraphics[width=1.8in]{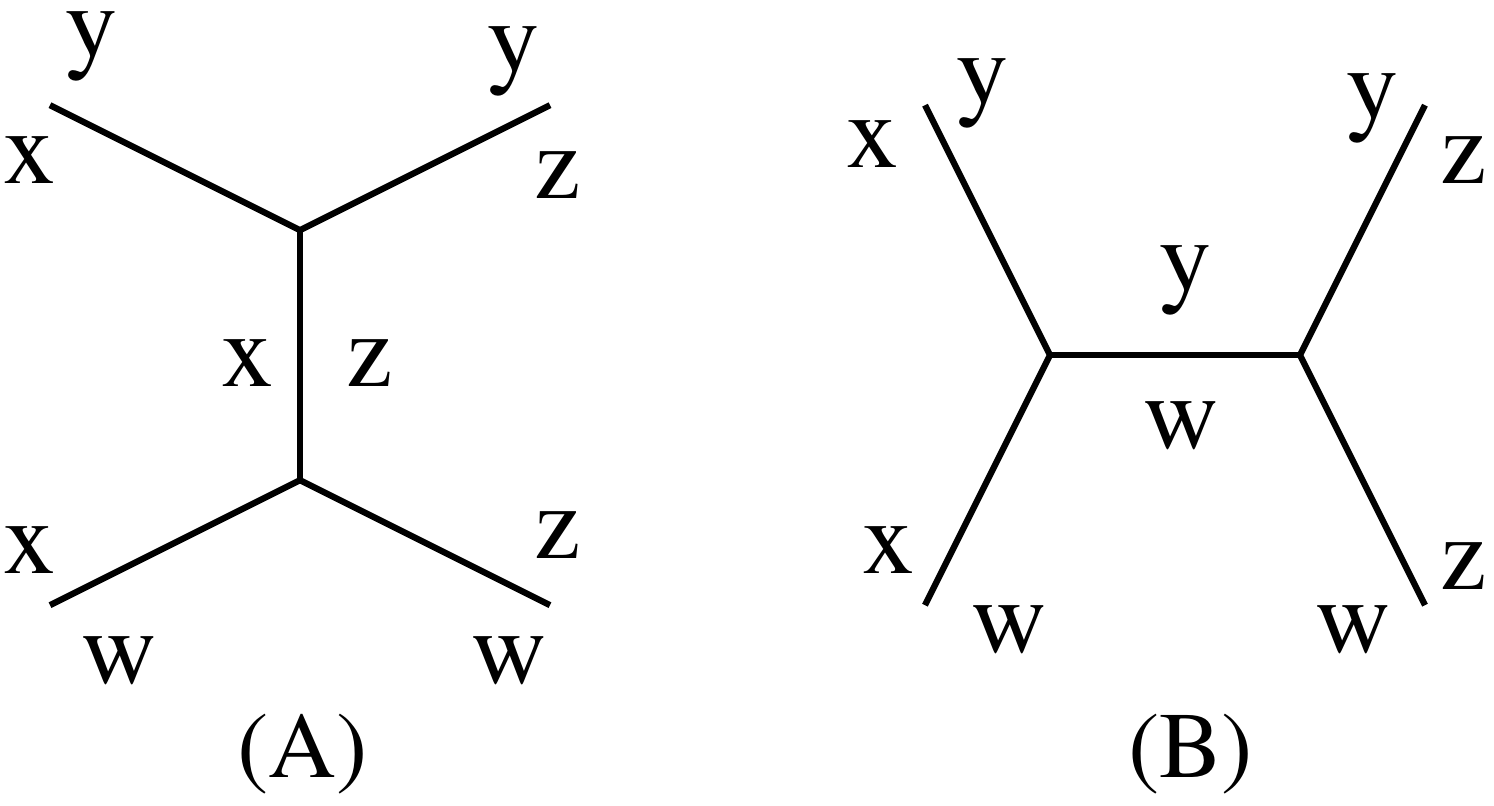}
\end{center}
\caption{Labeled IH move}
\label{assoc}
\end{figure}

\begin{figure}[htb]
\begin{center}
\includegraphics[width=4.5in]{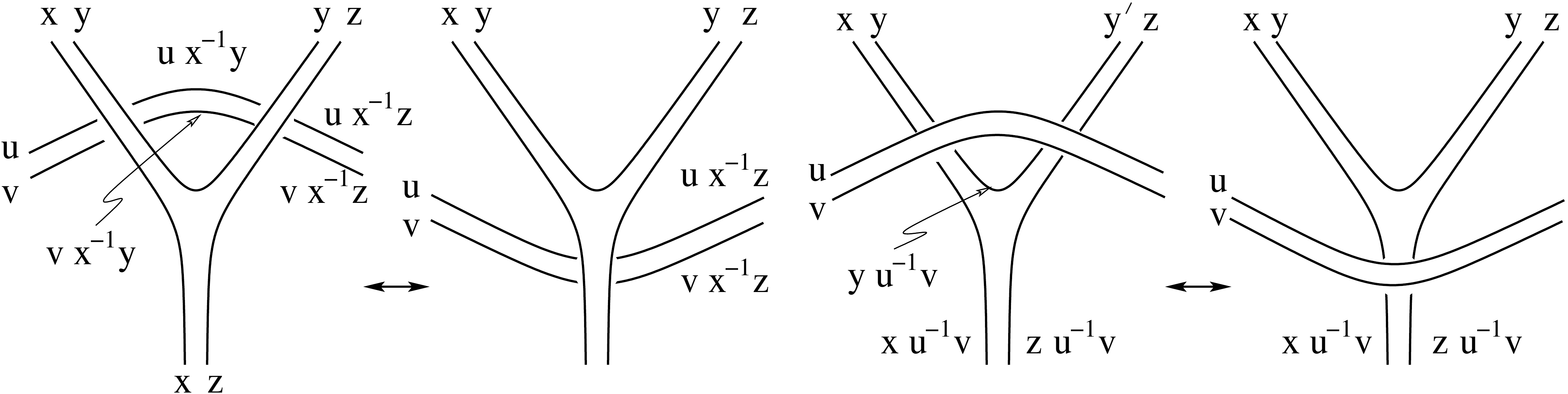}
\end{center}
\caption{YI and IY moves}
\label{BFdoublecolor}
\end{figure}

\begin{lemma} \label{lem:fund} 
The fundamental heap $h(S)$ is well defined, that is, 
the isomorphism class of the group heap $h(D)$ is independent on the choice of $D$.
\end{lemma}

\begin{proof}
Applying the moves for the spines of surface ribbons in \cite{Matsu}, we check the invariance under the moves listed in Figure~\ref{moves}.
Of these, Reidemeister moves RII and RIII, as well as the CL move, are checked in \cite{SZframedlinks}. A diagram for checking RIII is depicted in Figure~\ref{heaptypeIIIcocy}. 
In this figure, for example, using relations at crossings,
 the generator assigned on  the bottom right string,
is expressed in terms of $x,y,z,u,v$  as $xy^{-1}zu^{-1}v$
and $( x u^{-1} v) (y u^{-1} v)^{-1} (z u^{-1} v)$,  in the left and right figures, respectively, 
and they coincide.

The IH move does not involve any change of relations, and the presentation does not change.
The way the connected components of arcs receive consistently the same letters is depicted in Figure~\ref{assoc}. 
The remaining moves, YI and IY, are also checked diagrammatically as depicted in Figure~\ref{BFdoublecolor}, as desired. 
 We note that if we write $T(x,y,z)=x y^{-1} z$ as a ternary operation, 
then the YI move requires $T( T(u, x, y), y, z)=T(u, x,z)$ (left of Figure~\ref{BFdoublecolor}), 
and the IY move requires $T(y, u, v), v, u)=y'=y$ (right of Figure~\ref{BFdoublecolor}),
both of which hold for the group heap operation.  
These properties of heaps are also used in the proof of Lemma~\ref{lem:color}, accordingly.
\end{proof}

	Recall that a group $G$ is said to be finitely presented if there exists a presentation of $G$ with a finite number of generators and a finite number of relators. 
		From the definition, we have that the fundamental heap of a surface ribbon is the group heap of a
	finitely presented  group.
	From  the definition  
we have the following as well.

\begin{lemma} 
\label{lem:disjoint}
	Let $S_1, S_2, \ldots , S_\nu$ be  surface ribbons with fundamental heaps $h(S_i)$, $i= 1,\ldots , n$. Then the surface ribbon $S = \sqcup_i S_i$, a split sum (disjoint union), has fundamental heap $h(S) \cong h(S_1)*h(S_2)*\cdots * h(S_\nu)$.
\end{lemma}

The following is a generalization of the corresponding result in \cite{SZframedlinks}, which  was proved for framed links.

\begin{theorem}\label{thm:heap}
Let $S=S_1 \cup \cdots \cup S_\nu$ be a surface ribbon written as the union of connected components.
Then $h(S)\cong F_\nu * \hat{h}(S)$ for some group $\hat{h}(S)$, where $F_\nu$ denotes the free group of rank $\nu$. 
\end{theorem}

\begin{proof}
Let $x, y$ be a pair of generators assigned to a single ribbon $R$.
We call the elements  $x^{-1} y$ and $y^{-1}x$ {\it ribbon terms}, and a word in ribbon terms  a {\it ribbon word}.

\begin{figure}[htb]
\begin{center}
\includegraphics[width=2.2in]{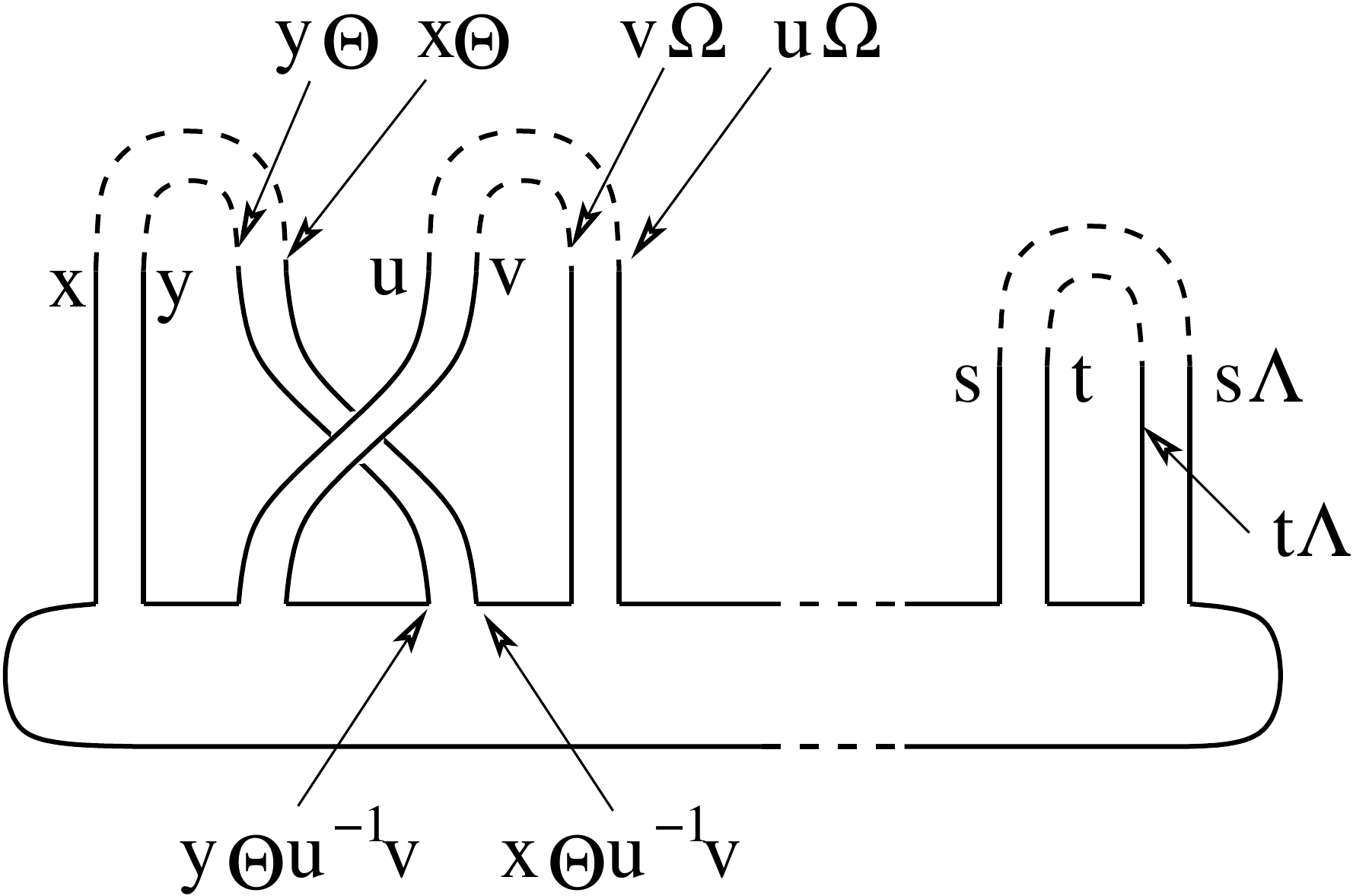}
\end{center}
\caption{Base of surface ribbons}
\label{base}
\end{figure}

First assume that $S$ is connected.
Any  connected surface ribbon can be put in a ``standard form'' as depicted in Figure~\ref{base}, 
where dotted arcs are ribbons connecting them,
and they may be knotted and linked  with each other.
In Figure~\ref{base}, the left portion indicates a pair of crossed handles (ribbons) that contribute to the genus by one, and the right portion indicates a single trivial handle that contributes to the number of connected components by one. Although in general there are multiple numbers of these, we make computations for the surface having one for  each type, as the general case is similar.
Assigned letters are indicated in the figure, where greek letters indicate ribbon words.

The presentation of $h(S)$ with arcs as generators and relations at crossings as defined 
in Definition~\ref{def:fund} is modified as follows, that represent an isomorphic group.
For this proof, we choose a set of generators to be the (letters assigned to) arcs ${\mathcal A}_0$ of ribbons connected 
to the base separately, in addition to arcs away from the base, and a set of relators ${\mathcal T}_0$ to be equalities when two feet of ribbons are adjacent 
at the base, in addition to relations derived from crossings as in Definition~\ref{def:fund}. For example, the arc with letter $y$ at the left most ribbon foot in Figure~\ref{base} is adjacent to the arc with letter $u$ at the second foot, so that a relation $y=u$ is part of ${\mathcal T}_0$.
From Figure~\ref{base}, one obtains a set of relations:
(1) $y=u$, (2) $v=y\Theta u^{-1} v$, (3) $x \Theta u^{-1}v= v \Omega$, (4) $u\Omega = s$, (5) $t=t \Lambda$, and 
(6) $s \Lambda=x$ from connected segments. Thus these relations are contained in ${\mathcal T}_0$.
This is a starting presentation $\langle\, {\mathcal A}_0 \mid {\mathcal T}_0  \, \rangle$.

If a pair of letters $x,y$ are assigned to the boundary arcs of a ribbon, then we add a
new generator $\alpha$  that is a ribbon term, and add a relation $\alpha = x^{-1} y $ that is a defining relation of the ribbon term.
Perform this process to obtain a new presentation 
$\langle \, {\mathcal A}_1 \mid {\mathcal T}_1 \, \rangle$.

When one traces the boundary curve labeled $x$ along the dotted line and encounter 
the first crossing as in Figure~\ref{buildingblocks} (A), together with the generator $z$ assigned on the arc on the other side of a pair of arcs labeled $u$ and $v$, we have a relation $z=x u^{-1}v=x \beta$, where $\beta=u^{-1}v$ is a ribbon term.
By this relation, $z$ is eliminated from the set of generators and replaced by $x \beta$, 
and in all relations having  $z$ in them are replaced by those with $z$ substituted by $x\beta$.
By this process, the generator $z$ is replaces by $x \beta$.
Similarly, $w$ is replaced by $y \beta$. 

Continuing this process, 
when the arc labeled $x$ reaches to the base of the surface as in Figure~\ref{base}
at the arc third from the left, it is labeled by $x \Theta u^{-1} v$, where $\Theta$ is a ribbon word and $u^{-1}v$ is a ribbon term $\beta$. The other arcs are similarly labeled in the figure.
The other relations at the base are changed as follows:
 (1)  $u = x\alpha$, 
 (2) $v=  y\Theta u^{-1} v= x \alpha \Theta \beta$, 
(3)  $(v=)  \Theta \beta \Omega^{-1} = \alpha \Theta \beta$,
 (4) replaces $s= x\alpha \Omega$, 
 (5) $\Lambda=1$ for ribbon word $\Lambda$, and  $t$ is replaced using a ribbon term $\gamma=s^{-1} t$ to be $t=s \gamma = x \alpha \Omega \gamma$, and 
 (6) $s=x \Lambda^{-1}$.
Thus we obtain a group presentation with all the generators original assigned to arcs are expressed by $x \Phi$ for some ribbon word $\Phi$.

In summary the original presentation of $h(S)$, 
$
\langle
\,
{\cal A}  \ | \ 
{\cal T} 
\,
\rangle 
$ as in Definition~\ref{def:fund} is replaced by 
a new resulting presentation 
$ \langle \,  x, {\mathcal B} \mid {\mathcal R} \,  \rangle$, where $x$ is a free generator, 
 $ {\mathcal B} $ is the set of generating ribbon terms corresponding to  ribbon segments,
and $ {\mathcal R} $ is a set of relations among ribbon words.
Hence $h(S)$ is written as $ \langle \,  x \ \rangle * \langle \ {\mathcal B} \mid {\mathcal R} \ \rangle = \Z * \hat{h}(S)$ where $ \hat{h}(S) = \langle \ {\mathcal B} \mid {\mathcal R} \, \rangle$.
The argument is repeated to higher genus and with more than one trivial bands.
The argument is also repeated, with a single free generator for each connected component $S_i$, $i=1, \ldots, \nu$, with one free generator for each component.
\end{proof}

\begin{definition}
{\rm 
By uniqueness of free product of groups~\cite{Scott-Wall}, the group $\hat{h}(S)$ in Theorem~\ref{thm:heap} is well defined up to isomorphism.
We call the group heap of $\hat{h}(S)$  the  {\it reduced fundamental heap} of $S$.
}
\end{definition}

\begin{definition}\label{def:rank}
{\rm 
For a surface ribbon $S$, the maximum rank $s$ of the free group factor $h(S)\cong F_s * G$ 
of the fundamental heap is called the {\it rank} of $h(S)$, or simply,  of $S$, and denoted by ${\rm rank} (S)$. 
}
\end{definition}

\begin{remark}
	{\rm 
		We observe that the rank of a surface ribbon is well defined, as a consequence of the uniqueness of free product factorization of groups \cite{Scott-Wall}.
	}
\end{remark}

\begin{definition}
	{\rm 
	Denote by $\mu(G)$ the minimum number of generators of a finitely generated group $G$. For a surface ribbon $S$
	let $\mu(S)$ denote $\mu(h(S))$. 
}
\end{definition}

From Theorem~\ref{thm:heap}, we have 
${\rm rank}(S) \geq  \nu$, where 
$\nu$ denotes the number of connected components of $S$. 
 In general the inequality is strict, as we see below.

\begin{figure}[htb]
\begin{center}
\includegraphics[width=2.2in]{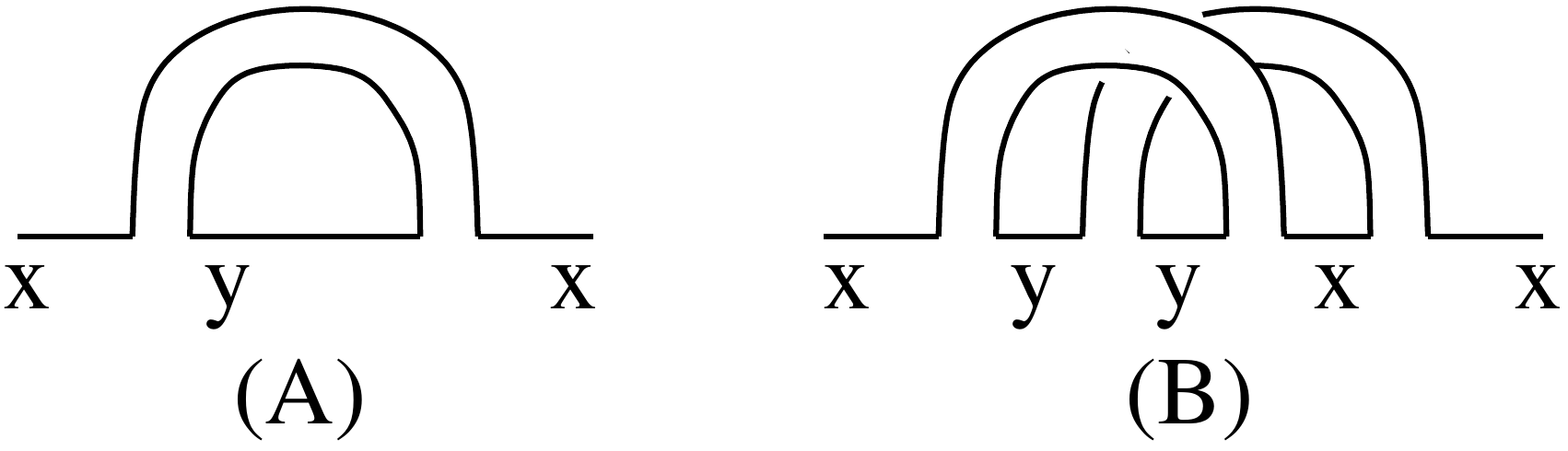}
\end{center}
\caption{Adding trivial  bands}
\label{trivribbon}
\end{figure}

\begin{example}\label{ex:loopband}
{\rm
A trivial single band $B_1$ and a trivial crossed band pair $B_2$ are depicted in (A) and (B) in Figure~\ref{trivribbon},
respectively.
If the two end points are closed by trivial arcs, then both result in the ribbon surface with 
the fundamental heap isomorphic to $F_2=\langle \, x,y \, \rangle$, the free group of rank two, as seen from the figure.
Let $S=(B_1)^m (B_2)^n$ denote the closure of concatenation of
$m$ trivial bands and $n$ pairs of  crossed band pairs. Then we have 
$h(S)\cong F_{m+n+1}$, the free group of rank $m+n+1$.

}
\end{example}

\begin{figure}[htb]
\begin{center}
\includegraphics[width=2.2in]{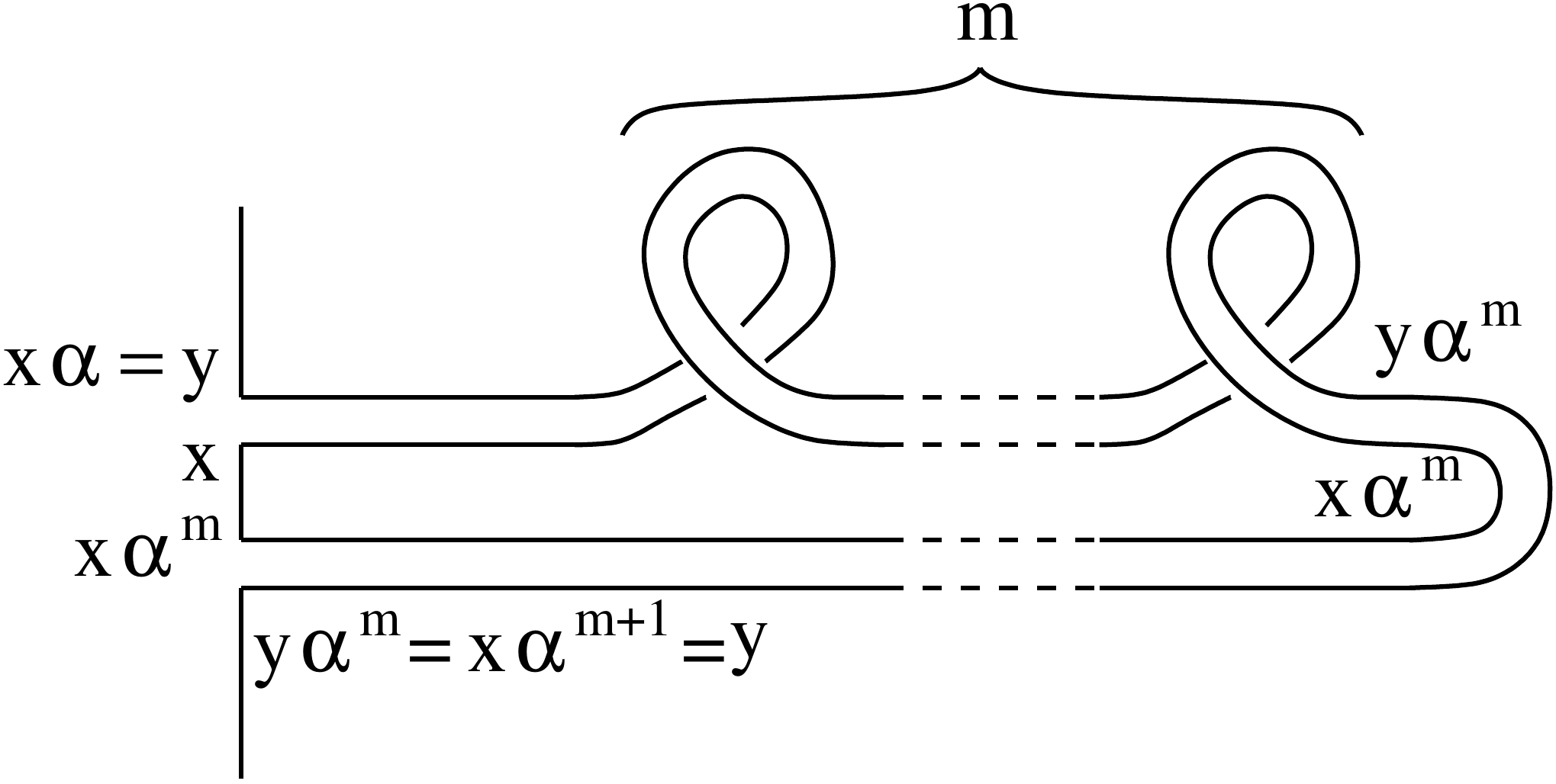}
\end{center}
\caption{Ribbon with loops}
\label{loops1}
\end{figure}

\begin{example}\label{ex:loops}
{\rm
A band with $m$ loops is depicted in Figure~\ref{loops1}.
At the end points, $x$ and $y$ are assigned as generators.
We set the ribbon element to be $\alpha=x^{-1}y$. 
Inductively at the right end of the arcs receive the labels $x \alpha^m$ 
and $y \alpha^m$ as depicted. This computation was done in \cite{SZframedlinks}. 
As depicted, we obtain a relation $\alpha^m=1$, and the bottom end receives the label $y$,
which coincides with the top label. 
If we close the end points, we obtain an annulus with the fundamental heap isomorphic to
$\langle \, x, \alpha \mid \alpha^m \, \rangle \cong F_1 * \Z_m$. 
If we concatenate copies of the band with $m_i$ loops vertically, $i=1, \ldots, n$, and close the end points, then  we obtain an $n$ punctured disk, denoted 
$ D(m_1, \ldots , m_n)$, such that 
$h(D( m_1, \ldots , m_n))\cong F_1 * \Z_{m_1} * \cdots * \Z_{m_n}$.
Furthermore, if we concatenate $k-1$ copies of the  trivial band in Figure~\ref{trivribbon} (A),
then we obtain an $n+k$ punctured disk $D( m_1, \ldots , m_n;k)$ such that 
$h(D( m_1, \ldots , m_n;k) ) \cong F_{k+1}  *  \Z_{m_1} * \cdots * \Z_{m_n}$.
Variations of this construction are found below in Example~\ref{ex:loopbandgenus}. 
}
\end{example}

Let ${\rm Ab}[G]$ denote the abelianization of a group $G$.
By  applying Lemma~\ref{lem:disjoint} to the abelianization of the fundamental heaps in Example~\ref{ex:loopband}  and Example~\ref{ex:loops}, 
we obtain the following.

\begin{proposition}\label{prop:ab}
For any finitely generated abelian group $A$, there exists a connected surface ribbon $S$ such that 
${\rm Ab}[h(S)]\cong A$. 
\end{proposition}

\begin{figure}[htb]
	\begin{center}
		\includegraphics[width=2.5in]{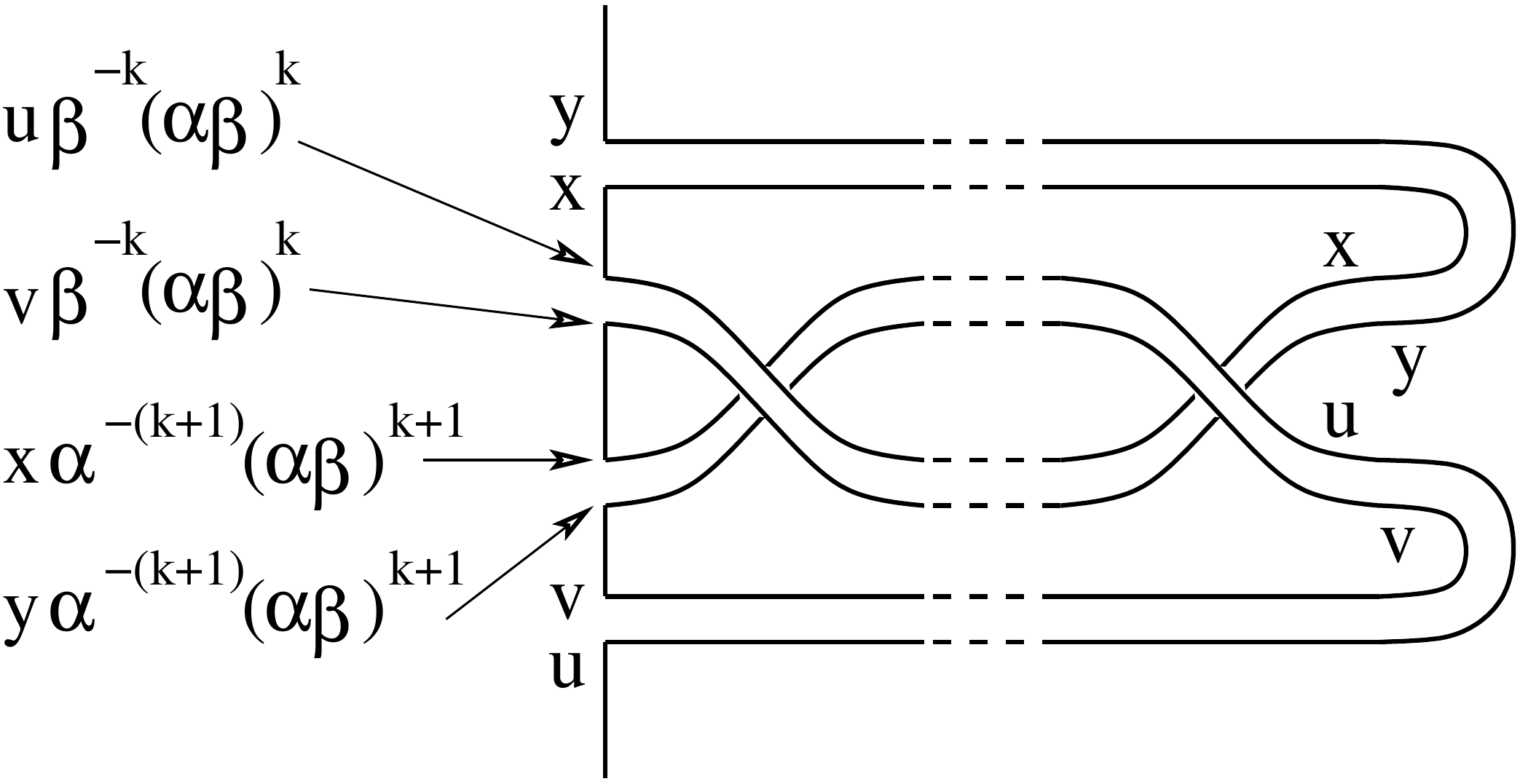}
	\end{center}
	\caption{Punctured torus }
	\label{highgenus}
\end{figure}

\begin{example}\label{ex:torussurface}
	{\rm 
		Let $T_1(k)$ denote the surface ribbon obtained by braiding two ribbons $(2k+1)$-times and closing them to give braided handles of a surface. This surface  is shown in  Figure~\ref{highgenus}, where top and bottom arcs are joined to give a punctured torus. 
			For the labels $x,y, u, v$ assigned to the arcs at the right of the figure, 
		the middle two bands at the left of the figure have arcs with the label 
		$$(u \beta^{-k} (\alpha \beta)^k , v  \beta^{-k} (\alpha\beta )^k ) \times (x \alpha^{-(k+1)} (\alpha \beta)^{k+1} ,
		y \alpha^{-(k+1)} (\alpha \beta)^{k+1})$$ as indicated in the figure.
		The fundamental heap with the open end points labeled $y$ and $u$ is generated by 
		$x,y,u,v$ with relations (1) $ x = u \beta^{-k} (\alpha \beta)^k$, 
		(2) $v  \beta^{-k} (\alpha \beta)^k=x \alpha^{-(k+1)} (\alpha \beta)^{k+1}$, and 
		(3) $  y  \alpha^{-(k+1)} (\alpha \beta )^{k+1} = v$,
		as read from the figure. 
		If we connect the arcs labeled by $y$ and $u$, we obtain a punctured torus, and an additional relation (4) $y=u$ holds. 

		Substituting (3) into (2) we obtain $(\alpha \beta)^{k+1}=\beta^{k+1}$.
		Then (1) implies 
		$$x=u \beta^{-k} (\alpha \beta )^k = u \beta^{-k} \cdot \beta^{k+1} (\alpha \beta)^{-1}
		= u \alpha^{-1} = u (x^{-1}y)^{-1}= uy^{-1} x,$$
		and we obtain the relation (4) $y=u$. Also (3) implies 
		$v = y \alpha^{-(k+1)} \cdot \beta^{k+1} = u \alpha^{-(k+1)}  \beta^{k+1} $ which implies
		$\beta = u^{-1} v = \alpha^{-(k+1)}  \beta^{k+1} $, hence $\alpha^{k+1} = \beta^k $. 
		Therefore the punctured torus $T_1(k)$ 
		has the fundamental heap
				$h(T_1(k))\cong \langle \, y, \alpha, \beta \mid \alpha^{k+1} = \beta^k , (\alpha \beta)^{k+1}=\beta^{k+1}
		\, \rangle$.

}
\end{example}

\begin{figure}[htb]
\begin{center}
\includegraphics[width=3in]{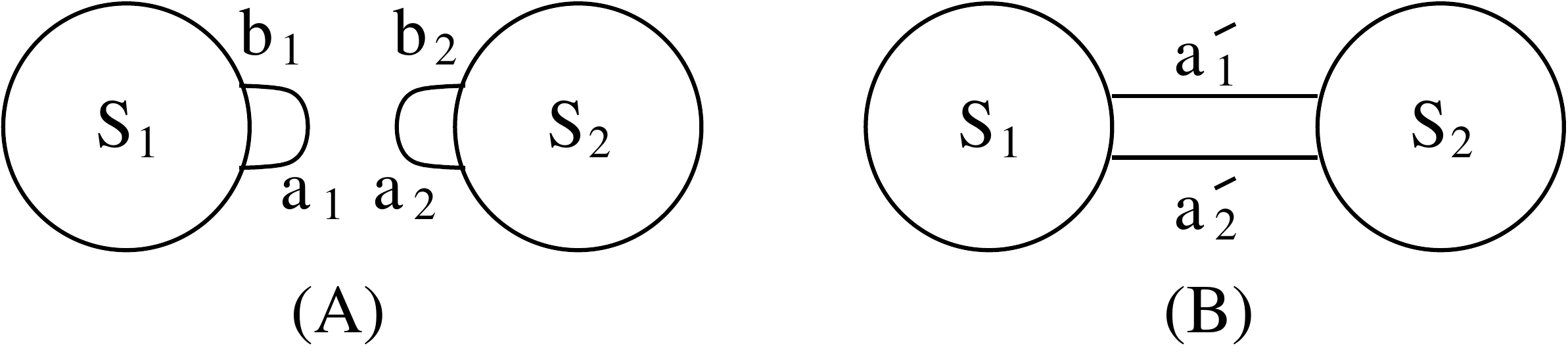}
\end{center}
\caption{Boundary connected sum}
\label{connect}
\end{figure}

The construction of concatenations in the above examples can be generalized to
{\it boundary connected sum} of surfaces, that are commonly used,  as follows. 
Let $S_1$, $S_2$ be surface ribbons. Regard them as embedded in disjoint 
balls as indicated in Figure~\ref{connect} (A). Specify a boundary component $b_i$ in $S_i$, 
$i=1,2$, and isotope a small portion of $b_i$  out of  the boundary of the ball as indicated in 
(A). Then connect the small portions by a short straight band as in (B). 
The resulting surface $S$ is denoted by $S_1 \natural_{(b_1, b_2)} S_2$, or 
$S_1 \natural S_2$ for simplicity. The ambient isotopy type of $S$ depends only on the choice of the boundary components $b_i$.

\begin{proposition}\label{pro:connected}
	Let $S_1$ and $S_2$ be  surface ribbons. Then there is an epimorphism $$\Gamma : \hat h(S_1\natural S_2) \longrightarrow \hat h(S_1)*\hat h(S_2).$$
	Moreover, $\Gamma$ can be extended to an epimorphism $\tilde \Gamma : h(S_1\natural S_2) \longrightarrow F_1*\hat h(S_1) * \hat h(S_2)$.
\end{proposition}

\begin{proof}
We proceed to give an epimorphism for presentations of the fundamental heaps of $S_1$, $S_2$ and $S_1\natural S_2$, where the generators $x_1$ and $x_2$ corresponding to arcs of the boundaries $b_1$ and $b_2$ 
of Figure~\ref{connect} (A), do not appear in the relators. 
This is possible, as in the proof of Theorem~\ref{thm:heap}, 
since along each boundary we have a relation of the form 
$x_i \Phi_i =x_i$
 for some ribbon word $\Phi_i$. 
  Let $a_1$ and $a_2$ be the arcs on $b_1$ and $b_2$  that are protruding out from $S_1$ and $S_2$  as in (A).
 The presentation of $\hat h(S_1)$ is obtained by moving counterclockwise along the boundary $b_1$, starting from the arc $a_1$ and imposing relators corresponding to the labeling conditions encountered along $b_1$. Finally, one obtains a relator of type $x_1 \Phi_1x_1^{-1}$, corresponding to the equality 
 $x_1 \Phi_1 = x_1$ 
 at $a_1$. Similarly, we have another relator $x_2\Phi_2x_2^{-1}$ for $\hat h(S_2)$. Clearly these relators are equivalent to $\Phi_1$ and $\Phi_2$, respectively.  
 After performing the boundary connected sum $S_1\natural S_2$, as in Figure~\ref{connect} (B), 
 let $a_1'$ and $a_2'$ be upper and lower boundary arcs
 in the bands connecting $S_1$ and $S_2$, respectively, and let $x_1'$ and $x_2'$ be generators assigned to them. 
 We proceed counterclockwise along the boundary 
 starting from $a_1'$ and read off relations at crossings. 
  We meet all the crossings that gave the presentation of $\hat h(S_1)$ obtaining all the same relators until we reach the lower portion $a_2'$ 
  of the band, 
  and obtain $x_1' \Phi_1 = x_2'$. 
 We proceed counterclockwise along the boundary $b_2$ of $S_2$. 
 The label of the arc outcoming from $S_2$ on top of the band connecting $S_1$ and $S_2$ now is given by
 $x_2' \Phi_2 = x_1'$, hence $x_1' \Phi_1 \Phi_2 = x_1'$. 
   As a consequence we have a presentation of $\hat h(S_1\natural S_2)$ where the generators are the union of the generators of $\hat h(S_1)$ and $\hat h(S_2)$, while all the relators are  obtained by the union of all relators of $\hat h(S_1)$ and $\hat h(S_2)$ but $\Phi_1$ and $\Phi_2$, which are now combined into a new relator $\Phi_1\Phi_2$. Let us denote $y_i$, $i = 1, \ldots , n$ the generators of $\hat h(S_1)$, and $z_j$, $j = 1,\ldots , m$ the generators of $\hat h(S_2)$. We distinguish the generators and relators of $\hat h(S_1\natural S_2)$ from those of $\hat h(S_1)$ and $\hat h(S_2)$ by introducing a ``tilde'' symbol on top. The respective relators are named $R_k$, $\Phi_1$ and $Q_t$, $\Phi_2$. We define the map $\hat h(S_1\natural S_2) \cong \langle \, \tilde y_i,\tilde  z_j\  |\ \tilde R_k, \tilde Q_t, \tilde \Phi_1\tilde \Phi_2 \, \rangle$ into $\hat h(S_1)*\hat h(S_2) \cong \langle \,  y_i, z_j\  |\ R_k, Q_t, \Phi_1, \Phi_2 \, \rangle$ by sending $\tilde y_i$ to $y_i$ and $\tilde z_j$ to $z_j$ for all $i$ and $j$. The map is well defined because all the relators $\tilde R_i$ and $\tilde Q_t$ are mapped to relators without tilde, hence vanish, while $\tilde \Phi_1 \tilde \Phi_2$ maps to the product $\Phi_1\Phi_2$ which vanishes as well, since $\Phi_1$ and $\Phi_2$ 
 are relators 
  in the free product separately.
  
The second statement is obtained by mapping the free generator of $h(S_1\natural S_2)$, obtained from Theorem~\ref{thm:heap}, onto the free factor $F_1$.  
\end{proof}

\begin{remark}\label{rmk:epiciso}
	{\rm 
	We observe that if $S_1$ and $S_2$ are  surface ribbons such that the presentations of $\hat h(S_1)$ and $\hat h(S_2)$ admit no nontrivial relator of type $\Phi_i$ in the notation of Proposition~\ref{pro:connected}, then the proof of the proposition gives that $\Gamma$ is an isomorphism of heaps.
}
\end{remark}

\begin{example}
	{\rm 
 Proposition~\ref{pro:connected} and Remark~\ref{rmk:epiciso} apply directly to the surfaces of Example~\ref{ex:loopband}, since the surfaces considered can be constructed as boundary connected sums of  surface ribbons whose (reduced) fundamental heaps are free. 
}
\end{example}

\begin{example}
	{\rm 
		Let $D_i := D(m_i)$, for $i = 1, \ldots , n$ denote a family of $n$  surface ribbons as in Example~\ref{ex:loops}, obtained by closing the ends of the diagram in Figure~\ref{loops1}. Then the presentation of each $\hat h(D_i)$ does not contain relators of type $\Phi_i$ in the notation of Proposition~\ref{pro:connected}, as the computation in Example~\ref{ex:loops} shows. In fact, observe that the labels of the arcs at the top and bottom of Figure~\ref{loops1} coincide, once the relator $\alpha^{m_i}=1$ is imposed. Then Proposition~\ref{pro:connected} and Remark~\ref{rmk:epiciso}, together with a simple inductive argument, imply that $h(D( m_1, \ldots , m_n)) \cong F_1*\hat h(D_1) * \cdots * \hat h(D_n)$. The computations in Example~\ref{ex:loops} show directly that $\hat h(D_i) = \Z_{m_i}$. This gives the fundamental heap of $h(D( m_1,\ldots, m_n))$ as expected. 
}
\end{example}

\begin{example}
{\rm 
	We let $T_g(k_1, \ldots, k_g)$ denote the boundary connected sum of $g$ surfaces $T_1(k_i)$ in Example~\ref{ex:torussurface}, where the boundaries used for the connected sum are chosen to be the base of each copy $T_1(k_i)$ (in standard form, the left most arcs in Figure~\ref{highgenus}). Then, the fundamental heap of $T_g(k_1, \ldots, k_g)$ is $h(T_g(k_1,\ldots, k_g)) = *_{i=1}^gh(T_1(k_i))$,
  obtained from the computation for the torus 
  in Example~\ref{ex:torussurface} 
  by applying Proposition~\ref{pro:connected} and Remark~\ref{rmk:epiciso}, since the relator coming out of connecting arcs at the left of the figure  for  each $T_1(k_i)$ is trivial. 
  
}
\end{example}

 In the following example we see that Proposition~\ref{pro:connected} can be used to determine the minimum number of generators of the fundamental heap.
 
\begin{example}
	{\rm 
		Grusko's Theorem implies  that $\mu(G_1*G_2) = \mu(G_1) + \mu(G_2)$, see \cite{Scott-Wall}.
		From the epimorphism of Proposition~\ref{pro:connected} we see that 
		 $ \mu(S_1\natural S_2)+1 \geq \mu(S_1) + \mu(S_2) $.
		
		In particular, for those surfaces for which $\tilde \Gamma$ is an isomorphism, we have an equality 
		$\mu(S_1\natural S_2) +1 = \mu(S_1) + \mu(S_2) $. 
		This formula can be applied successively for boundary connected sums of more than two  surface ribbons.
		This is in fact the case with the surfaces of Examples~\ref{ex:loopband},~\ref{ex:loops} and~\ref{ex:torussurface}. We have that $\mu((B_1)^m(B_2)^n) = m+n+1$ and $\mu(D( m_1,\ldots , m_n)) = n+2$. Pertaining to the surface $T_g(k_1,\dots,k_g)$ of Example~\ref{ex:torussurface}, we obtain that $\mu (T_g(k_1,\dots,k_g)) = g+1$ or $\mu(T_g(k_1,\dots,k_g)) = 2g+1$, depending on whether $\mu (\langle \, y, \alpha, \beta \mid \alpha^{k+1} = \beta^k , (\alpha \beta)^{k+1}=\beta^{k+1}
		\, \rangle) = 2, 3$. Set $H= \langle \, y, \alpha, \beta \mid \alpha^{k+1} = \beta^k , (\alpha \beta)^{k+1}=\beta^{k+1}
		\, \rangle$ for simplicity. It is well known that the number of generators of $G/[G,G]$ provides a lower bound to the minimum number of generators of a group $G$. From the abelianization of $H$, namely $H/[H,H] = \Z\oplus \Z_{k+1} \oplus \Z_k$, we see that $H$ has at least three generators. It follows that $\mu(T_g(k_1,\dots,k_g)) = 2g+1$. We observe that this is independent of the crossings $k_i$ of each torus surface component of $T_g(k_1,\dots,k_g)$.
	}
\end{example}

\begin{remark}\label{rmk:monic}
	{\rm 
It is not clear whether the epimorphism $\Gamma$ of Proposition~\ref{pro:connected} is also 
 always monic.
This situation affects a formula for the cocycle invariant under boundary connected sum.
If $\Gamma$ is monic, then the extra term 
in the cocycle invariant considered  in Remark~\ref{pro:connectedcocy}  would vanish.
}
\end{remark}

%

\subsection{Adding a twisted band and realizations of the fundamental heap}

In Figure~\ref{adloop}, a local operation of a surface ribbon is depicted.
On the left, a single ribbon portion of a surface ribbon $S$ is depicted.
A twisted band is attached to the ribbon as depicted in the figure to obtain a new surface ribbon $S'$.
The symbols involving $\psi$ will be used later.
We call this operation an {\it addition of a (positively) twisted band}. An addition of a negatively twisted band is similarly defined with the opposite crossing information for the added band. 
Note that the number of connected components of the surface does not change
under this operation, and 
the number  of connected components  of the boundary curves changes by one under this operation;
if a band is attached to distinct boundary components, then the number reduces by one, and the opposite otherwise.

\begin{figure}[htb]
\begin{center}
\includegraphics[width=2in]{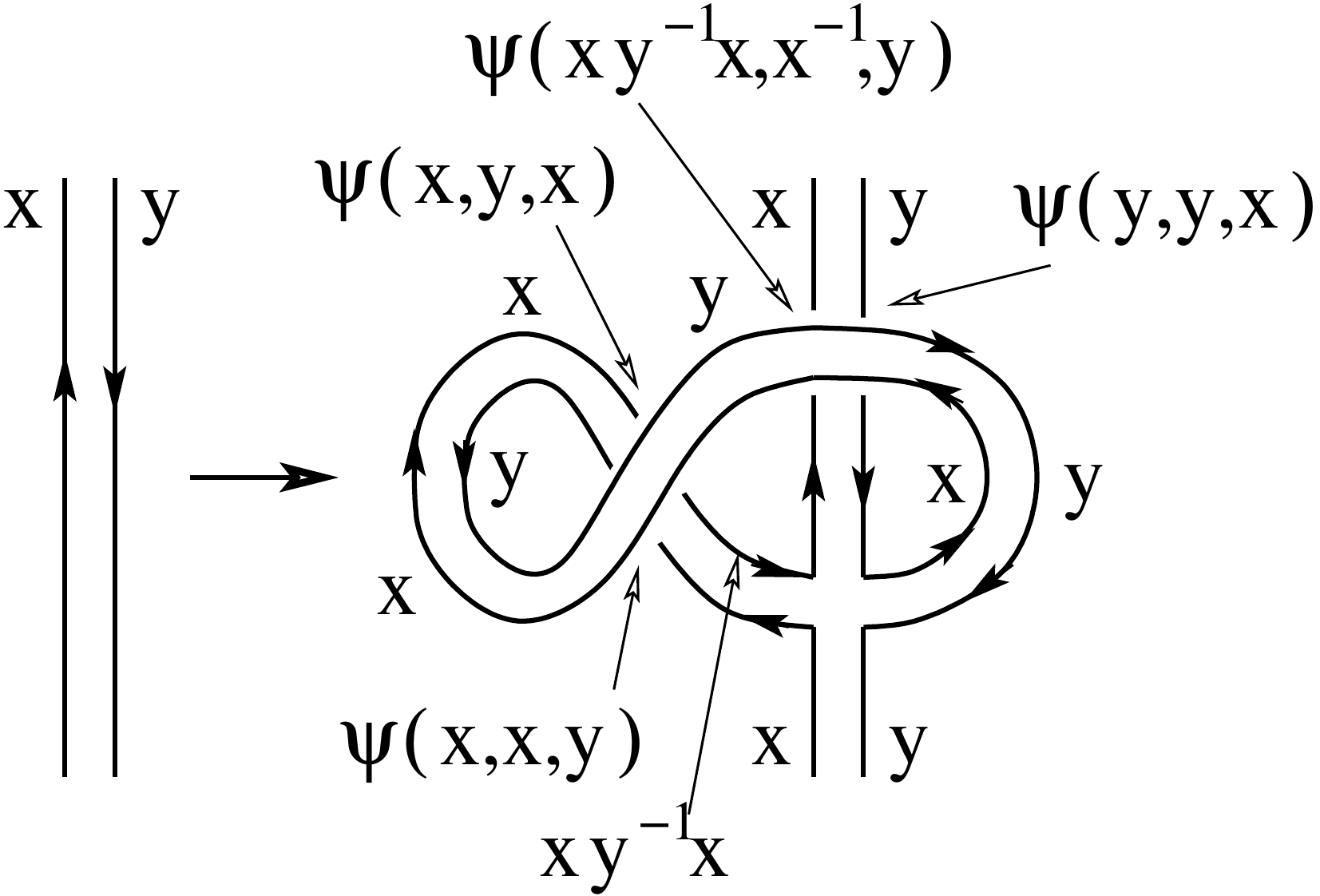}
\end{center}
\caption{Adding a twisted band}
\label{adloop}
\end{figure}

\begin{lemma} \label{lem:adloop}
Let $S'$ be a surface ribbon obtained from $S$ by adding  a twisted band.
Then we have $h(S')\cong h(S)$ and $\hat{h}(S') \cong \hat{h}(S)$. 
\end{lemma}

\begin{proof}
As depicted in Figure~\ref{adloop}, the generators assigned on the arcs of the added band 
are uniquely expressed by $x$ and $y$, the generators of the ribbon of $S$ where
the band is attached, by means of relations at crossings. 
This shows that the presentations of the diagrams of $S$ and $S'$ define isomorphic groups.
Since the number of connected components of $S$ and $S'$ is the same as noted above,
Theorem~\ref{thm:heap} and the uniqueness of free product of groups~\cite{Scott-Wall} 
imply  that $\hat{h}(S') \cong \hat{h}(S)$. 
\end{proof}

\begin{example} \label{ex:loopbandgenus}
{\rm 
Let $D(n,k)$ be a disk with $n$ bands with $m_i$ loops, $i=1, \ldots, n$, and $k$ trivial bands, as in Example~\ref{ex:loopband}.
Add $\ell (\leq n+k) $ twisted loops as in Figure~\ref{adloop} onto $g (>0)$ distinct 
bands among $n+k$ bands attached to the disk, to obtain a surface $S(g, n+k)$ of genus $g$ and the number of boundary component $b:=(n+k +1)-g$. 
This surface  $S(g, n+k)$ has the same fundamental heap as that of $D(n,k)$ in Example~\ref{ex:loopband},
but has a non-zero genus $g$.
}
\end{example}

Proposition~\ref{prop:ab}
is restated as follows, using Example~\ref{ex:loopbandgenus}.

\begin{proposition}\label{prop:ab2}
 Let $A = \Z^{\oplus (k+1)} \oplus [\oplus_{i=1}^n \Z_{m_i}]$ be a finitely generated abelian group. Then, for any $g, b \in \Z_{\geq 0}$ such that g+b = n+k+1, there exists a connected surface ribbon $S(g,b)$ of genus $g$ and $b$ boundary components such that ${\rm Ab}[ (h(S(g,b)) ] \cong A$. 
 \end{proposition}

\begin{proposition} \label{prop:realize}
Let $S$ be a connected surface ribbon with the genus $g(S)$ and the number of 
boundary components $b(S)$ such that $b(S) >1$. 
 Then for any  integers $g' \geq 0$ and $ b' >0$ such that $g' \geq g(S)$  and $b' \geq  b(S) - (g' -  g(S))  $, there exists a surface ribbon $S'$ 
with ${h}(S')\cong {h}(S)$ such that  $g(S')=g'$ and $b(S')=b'$. 

In particular,  for any $S$ with $b(S)>1$ and any $\chi ' \leq \chi(S)$ which denotes the Euler characteristic, there exists $S'$ such that  ${h}(S')\cong {h}(S)$ and   $\chi(S') =\chi'$. 

If $b(S)=1$, the statement holds for any  $g' \geq g(S)$ and $b'  \geq   b(S)  $, and for any  $\chi ' < \chi(S)$.

The statements hold for $\hat{h}$ as well. 
\end{proposition}

\begin{proof}
We show that for any $S$ 
such that $ b(S ) -1>0$
there exists $S'$ and $S''$ with $h(S'), h(S'')  \cong h(S)$ such that:
{\rm (i)}  $g(S')=g(S)$ and $b(S')=  b(S) + 1 $,  and 
{\rm (ii)}  $g( S'' )=g(S)+1$ and $ b(S'') = b(S ) -1 $.
For a given $g'$ and $b'$ as stated, then, we apply Case (ii) to obtain 
$S''$ such that $g(S'')=g'$ and $b(S'')=b(S) - (g'- g(S))$, and apply Case (i) to obtain 
$S'$ with the desired $g(S')$ and $b(S')$.

Let $S'$ be the surface ribbon obtained from $S$ by adding  a twisted band to the same component of the boundary curve   as in Lemma~\ref{lem:adloop}. Then by the lemma ${h}(S')\cong {h}(S)$ and the condition (i) is satisfied. 
If $S'$ is obtained by adding a twisted loop to two distinct components, then (ii) is satisfied.

In Case (i), we have $\chi(S')=2 - b(S') - 2g(S') = 2 - (b(S) + 1 ) - g(S)=\chi(S) -1$, and in Case (ii), 
we have $\chi(S')=2 - b(S') - 2g(S') = 2 - (b(S) - 1 ) - 2 ( g(S) + 1) = \chi(S) -1$, so that the statement for $\chi$ holds. 

Alternatively, attaching a twisted band corresponds, homotopically, to attaching a loop, hence it contributes to $-1$ to the Euler characteristic, and we have $\chi(S')=\chi(S) -1$. 

If $b(S)=1$, then Case (ii) in the proof cannot be performed. 
If Case (i) is performed to $S$, we obtain $S'$ with  ${h}(S')\cong {h}(S)$  such that $g(S')=g(S)$ and $b(S')=  b(S) + 1 =2$, and $\chi(S')=\chi(S)-1$. 
If we perform Case (ii) to $S'$, we obtain $S''$ with  ${h}(S'')\cong {h}(S)$ 
such that $g(S'')=g(S)+1$,  $b (S'')= b(S')-1 =1$ and $\chi(S')=\chi(S)-1$. 
Hence the statements for $b=1$ follow.

The statement for $\hat h$ follows from Theorem~\ref{thm:heap} and the uniqueness of the free product.
\end{proof}

\begin{proposition}\label{prop:maxEuler}
	Let $S$ be a surface ribbon, then there exists a surface ribbon $S^*$ having maximum Euler characteristic among the surface ribbons with fundamental heap isomorphic to $h(S)$.
\end{proposition}
\begin{proof}
	Let us set $t := {\rm rank}(S)$, as in Definition~\ref{def:rank}.
	 Suppose, 
	for the sake of 
	contradiction, that such a maximum surface ribbon does not exist. Then, 
	there exists 
	a sequence $S_k$ with the properties that $h(S_k) \cong h(S)$ and $\chi(S_k) > \chi(S_{k-1})$ for $k \in \N$. For $k$ large enough, we have that $\chi(S_k) > t$. But then $\chi(S_k)$ is larger than $\chi(\sqcup_t \mathbb D)$, where $\mathbb D$ is the disk. Since $\sqcup_t \mathbb D$ is the surface with $t$ connected components with
	the largest
	Euler characteristic, it follows that $S_k$ has more than $t$ connected components. Consequently, from Theorem~\ref{thm:heap} it follows that ${\rm rank}(S_k) > t$, which is absurd, since $h(S_k) \cong h(S)$ by assumption, 
	and the rank of a ribbon surfaces is well defined by uniqueness of free product decomposition. 
\end{proof}

\begin{remark}
{\rm 
We say that $S'$ is obtained from $S$ by {\it removing a twisted band} if $S$ is obtained from $S'$ by adding a twisted band, as described above. 
Let $S$ and $S^*$ be as in Proposition~\ref{prop:maxEuler}.
We claim that we cannot remove a twisted band from $S^*$, 
 in the sense that there exists no surface ribbon $S'$ such that $S^*$ is obtained from $S'$ by adding a twisted band. 
In fact, if we could find such a surface ribbon $S'$,  then it follows (from proof of Proposition~\ref{prop:realize}) that $\chi(S')> \chi(S^*)$ and $h(S')  = h(S^*)$, contradicting that $\chi(S^*)$ is maximum.
}
\end{remark}

\begin{figure}[htb]
\begin{center}
\includegraphics[width=5.5in]{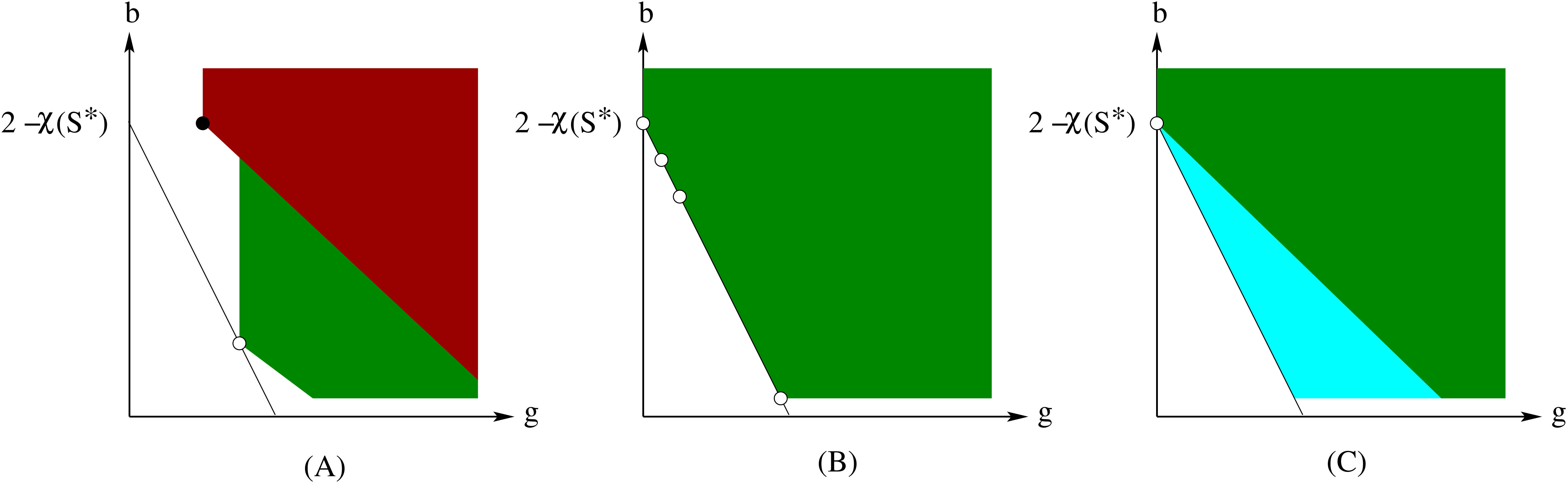}
\end{center}
\caption{Possible values of $(g, b)$}
\label{gb}
\end{figure}

We analyze possible values of genus $g$ and the number of boundary components $b$
for a given $h(S)$.
We assume that $S$ is a connected surface ribbon, since  the same argument can be applied to each component for disconnected surface ribbons. 
Possible values of $(g(S'), b(S') )$ with $h(S')\cong h(S)$ 
are schematically depicted in Figure~\ref{gb} (A). 
The axes represent $g$ and $b$ as indicated.
 The value $(g(S), b(S))$ of a given surface ribbon $S$ is 
indicted by a black dot located at the left top. The possible range of $(g,b)=(g(S'), b(S'))$ for some $S'$ 
guaranteed by Proposition~\ref{prop:realize} is  depicted by (red) dark shaded region. 
The region is bounded from the left by the line $g=g(S)$, and below by the lines
$b-b(S)=-(g-g(S))$ and $b=1$. 

 In Figure~\ref{gb} (A), the white circle on the line $b=(2-\chi(S^*)) - 2g $ represents $(g(S^*), b(S^*))$.
 The line represents the formula
 $\chi=2-b-2g$. Since $\chi(S^*)$ is maximum, the $y$-intercept $2- \chi(S^*) $ is minimum among all lines
 though $(g(S'), b(S'))$ with $h(S')\cong h(S)$. 
 Hence all possible values of $(g,b)$ are bounded on the left by $g=0$,  and below by 
$b= (2 - \chi(S^*)) - 2g $  and $b=1$.

\begin{example}\label{ex:free}
{\rm
We consider surface ribbons $S$ with $h(S)\cong F_k$, the free group of rank $k>0$. 
By Example~\ref{ex:loopband},  the surface ribbon obtained by boundary connected sum
$S=(B_1)^m(B_2)^n$ of $m$ trivial bands $B_1$ and $n$ trivial crossed band pairs $B_2$ 
have $h(S)\cong F_{m+n+1}$. 
In Figure~\ref{gb} (B), the left top white circle represents $S=(B_1)^{k-1}$, with $b(S)=k$ and $g(S)=0$. 
Inductively, the white dots represent $(B_1)^{k-3}(B_2)^1, (B_1)^{k-5}(B_2)^2, \ldots$ along the line 
$b=k-2g$. The bottom point on the line  is $S=(B_2)^{k/2}$ with $(g(S), b(S))=((k/2), 1)$ if $k$ is even,
and is $S=(B_1)^1(B_2)^{(k-1)/2}$ with $(g(S), b(S))=((k-1)/2, 2)$ if $k$ is odd (in this case there is no point on the line with $b=1$). 
The union of the regions bounded by $g=0$, $b-b(S)=-(g-g(S))$ and $b=1$  described above 
for these points $S$ 
cover the integral points bounded by $g=0$, $b=k - 2g $  and  $b=1$ as represented by the green shaded region in Figure~\ref{gb} (B). 
The green region in (A) sweeps out that of (B) over all white points on the line of slope $-2$. 

 }
\end{example}

\begin{example}
{\rm
Next we consider $S=D(  m_1, \ldots, m_n)$, $m_j>1$, in Example~\ref{ex:loops},
where $h(S)\cong \Z * \Z_{m_1} * \cdots *\Z_{m_n}$. 
The white point in the figure represents $S$, with $g(S)=0$ and $b(S)=n+1$. 
The region of integral points representing $S'$ with $h(S')\cong h(S)$  by Proposition~\ref{prop:realize}
is represented by the green region in Figure~\ref{gb} (C).

We compare this region in (C) with the region in (B). 
The region (B) was realized through existence of points on the line $b=(2 - \chi(S^*)) -2g$.
In this case of $S=D(  m_1, \ldots, m_n)$, even if we 
assume that this $S$ is of maximum Euler characteristic,
we do not have surface ribbons that  correspond to integral points between $b=(2 - \chi(S^*)) -2g$ and $b=(2 - \chi(S^*)) -g$, above $b=1$, represented by the region shaded in light blue in (C). 
For example, for $n=2$, since $b=3$, the point $(g,b)=(1,1) $ is on the line $b=(2 - \chi(S^*)) -2g$
but below the line $b=(2 - \chi(S^*)) -g$, and we do not know if there exists a surface ribbon realizing this point.
}
\end{example}

For the $y$-intercept  $2 - \chi(S^*)$ in Figure~\ref{gb}, we have the following bound.

\begin{proposition}\label{prop:rank}
	Let $S$ be a surface ribbon with $\nu$ connected components. Then we have 
	$$
		\mu(S) \leq 2\nu -\chi(S^*) 
	$$
	where $S^*$ is the surface ribbon with maximum Euler characteristic satisfying $h(S) \cong h(S^*)$.
\end{proposition}
\begin{proof}
	We assume that $S$ is connected, since an iteration of this case 
	gives the result in general. From the proof of Theorem~\ref{thm:heap}, after putting $S$ in standard form, we have a presentation of $h(S)$ with a copy of the free group on one generator, and the reduced heap, $h(S)=F_1 * \hat{h}(S)$. The latter has two generators corresponding to pairs of ribbon terms for each crossed handle pair, and one generator corresponding to a ribbon term for each trivial handle. Therefore, we have $2g + b = 2- \chi(S)$ generators. Consequently we find that $\mu(S) \leq 2-\chi(S)$. Since $S^*$ has $h(S^*) \cong h(S)$ by assumption, 	 it follows that  $\mu(S^*) = \mu(S)$. From $\chi(S^*) \geq \chi(S)$ and the previous inequality for $\mu(S)$ we obtain $\mu(S) \leq 2- \chi(S^*)$, and 
		 the proof is complete.
	\end{proof}


\subsection{A relation to the Wirtinger presentation}

The fundamental group of the complement of a spatial graph can be given by a presentation  from its oriented diagram in a manner similar to the Wirtinger presentation of knot groups, as outlined in \cite{MSW}.

\begin{figure}[htb]
\begin{center}
\includegraphics[width=2in]{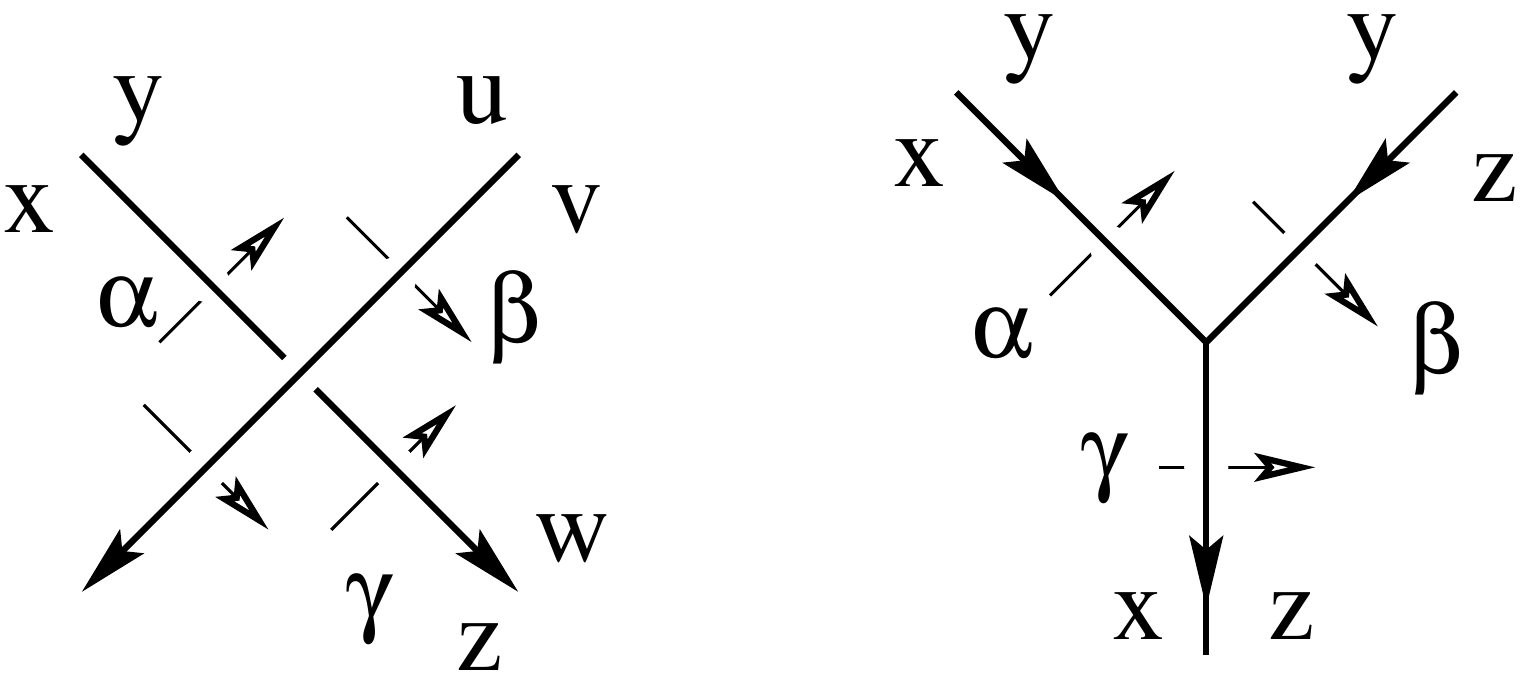}
\end{center}
\caption{Wirtinger relations at a crossing (left) and a vertex (right)}
\label{vertori}
\end{figure}

\begin{theorem}
For a surface ribbon $S$, let $\hat{h}(S)$ be the reduced fundamental heap,  $h(S)=F_{\nu}* \hat{h}(S)$.
Then there exists an epimorphism   $\lambda: \pi_1(S^3 \setminus S) \rightarrow \hat{h}(S)$.
\end{theorem}

\begin{proof}
Let 
$\Xi$
be the core trivalent graph of $S$. We specify arbitrary orientations on edges of $\Xi$, so that there are four possibilities of the orientations at each vertex.
The case of two-in, one-out is depicted in the right of Figure~\ref{vertori}.
In the figures, Wirtinger generators are depicted by short arcs behind oriented edges of 
$\Xi$. 

At a positive crossing depicted in the left of Figure~\ref{vertori}, 
we set the ribbon terms 
$\alpha=x^{-1}y$, $\beta=u^{-1}v$ and $\gamma=z^{-1}w$, where $x,y, u,v, z,w$ are generators of $h(S)$ assigned on the parallel boundary arcs of the surface ribbon $S$. 
The assignment that defines $\lambda$ is such that the Wirtinger generator of the arc corresponding to the arc
labeled by $(x,y)$ is assigned $\alpha$, and similar for the other arcs.
This is the same assignment defined in \cite{SZframedlinks} for framed links, in which 
the Wirtinger relations are verified using the diagram in the left of Figure~\ref{vertori}.
Indeed, since $z= x u^{-1} v $ and $w=y u^{-1} v$, one computes 
$$\gamma = z^{-1} w = ( x u^{-1} v )^{-1} (y u^{-1} v) = (u^{-1} v)^{-1} (x^{-1} y) ( u^{-1} v)= 
\beta^{-1} \alpha \beta, $$
which is a Wirtinger relation.
Negative crossings are checked similarly. The group $\hat{h}(S)$ is generated by 
ribbon terms of the form 
 $\alpha=x^{-1}y$, hence the image of $\lambda$ is in $\hat{h}(S)$.

It remains to show that the relation holds at trivalent vertices. 
For a vertex depicted in the right of Figure~\ref{vertori}, the relation in $\pi_1(S^3 \setminus S)$
is $\gamma=\alpha \beta$, and this holds for 
$\alpha=x^{-1}y$, $\beta=y^{-1} z$ and $\gamma=x^{-1} z$ as desired.
The other three types of orientations at vertices can be similarly checked.
\end{proof}

\subsection{Effect under stabilization}

In this section we describe a stabilization of 
surface ribbons and provide the effect of a stabilization on the fundamental heap.
It is known (e.g.~\cite{BFK}) that two Seifert surfaces of a link are related by a sequence of (de/)stabilizations and isotopy, where a stabilization means a 1-handle addition.

\begin{figure}[htb]
\begin{center}
\includegraphics[width=4in]{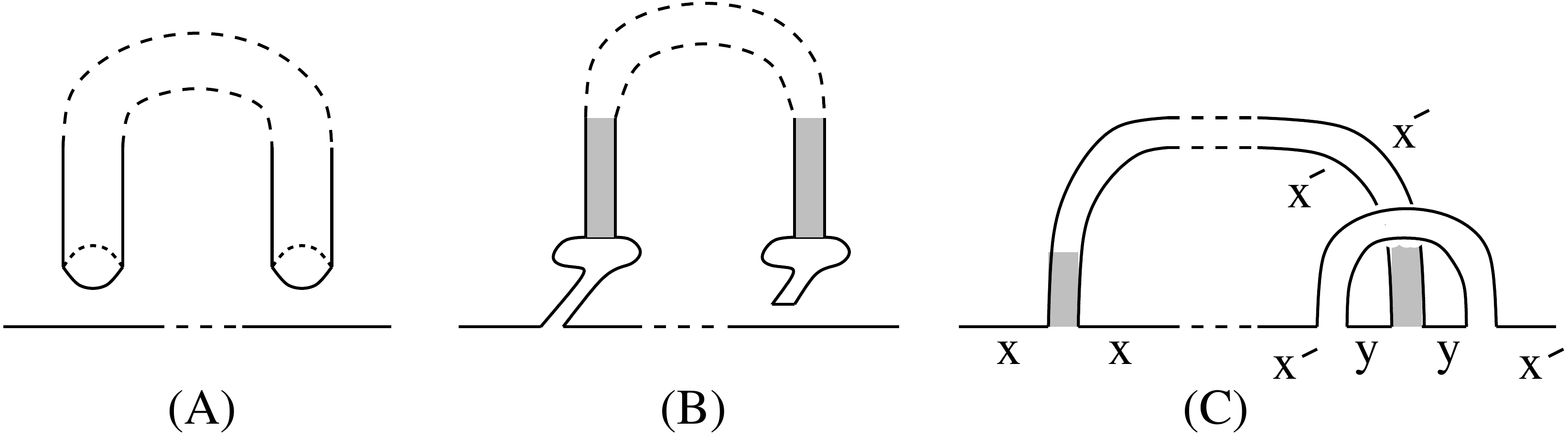}
\end{center}
\caption{Stabilization}
\label{stab}
\end{figure}

A 1-handle addition is depicted in Figure~\ref{stab} (A). In (B), a thin portion of the boundary is pushed towards the left foot of the handle, and wraps around the handle to obtain a thin ribbon that was a part of the handle. The boundary is pushed further along the handle to the right foot.
The pushed boundary curve stops short of reaching the boundary near the right foot of the handle as depicted.
By straightening and flipping, we obtain the surface in (C). 
In summary, 
an addition of  a pair of a long and a short trivial ribbons
as in (C) is regarded as a stabilization of a surface ribbon.

\begin{figure}[htb]
	\begin{center}
		\includegraphics[width=2.5in]{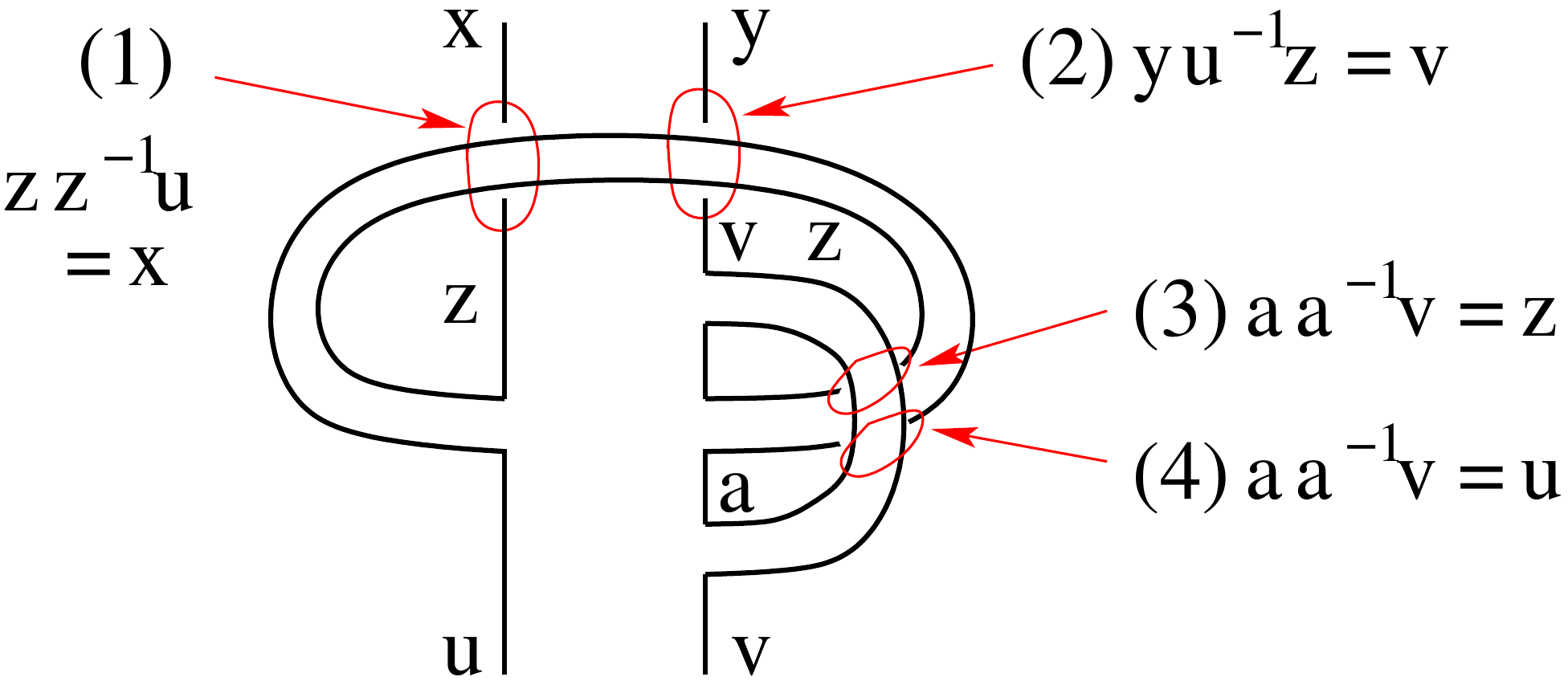}
	\end{center}
	\caption{Labels for a stabilization}
	\label{stablabel}
\end{figure}

\begin{proposition}\label{prop:stab}
	For any surface ribbon $S$, there exists another $S'$ obtained from $S$ by a sequence of
	stabilizations such that $h(S')$ is a free group. 
\end{proposition}

\begin{proof}
	We may assume that a given diagram of $S$ is connected. 
	In Figure~\ref{stablabel}, generators and relations for a stabilization at a ribbon is depicted.
	Each arc receives a generator as indicated, and 4 ribbon crossings as indicated by red circles 
	give rise to relations (1) through (4).
	From (1), (3) and (4)  we obtain that $x=u=v=z$, and from (2) we obtain these generators are equal to $y$. 
	When a stabilization is performed in this way to a vertical ribbon whose boundary curves receive generators $x$ and $y$,
	the effect of the stabilization is an additional relation $x=y$ and an additional free generator $a$.

\begin{figure}[htb]
\begin{center}
\includegraphics[width=2in]{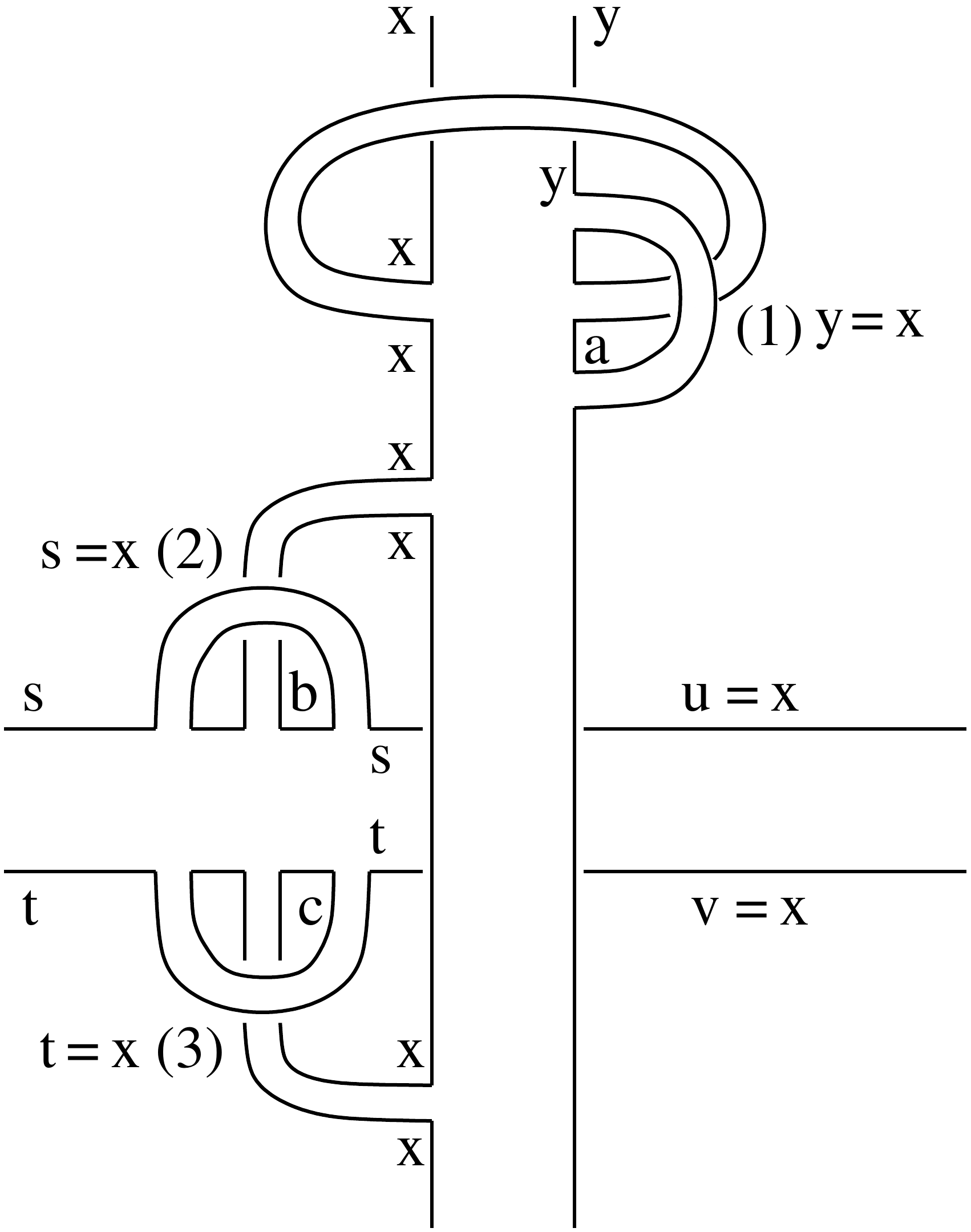}
\end{center}
\caption{Stabilization at  a crossing}
\label{stabcrossing}
\end{figure}

	When two other similar operations are performed at a crossing 
	 as depicted 	
 in
	Figure~\ref{stabcrossing}, the effect is that 
	corresponding to stabilizations labeled (1), (2), and (3), relations $y=x$, $s=x$ and $t=x$ are introduced as depicted,
	and the original relations for $u$ and $v$ imply $u=x$ and $v=x$ as well. 
	Three free generators are introduced, as indicated by $a$, $b$ and $c$ corresponding to small loops
	in the figure at (1), (2) and (3), respectively.
	Hence the effect of this process at this crossing is that all original generators assigned to arcs are equated, and three new free generators are introduced.
	
	By performing this procedure at every ribbon crossing of the diagram, we obtain $S'$ by stabilizations 
	such that $h(S')$ is a free group.
\end{proof}

\subsection{Realization problem for the fundamental heap of  surface ribbons}

Recall that, as observed above, fundamental heaps of surface ribbons are finitely presented by definition. From Theorem~\ref{thm:heap} it follows that for any  surface ribbon $S$, $h(S)$ contains a free factor. We show below that any finitely presented group can be realized as a fundamental heap after adding some free factor.
 
 \begin{figure}[htb]
	\begin{center}
		\includegraphics[width=3in]{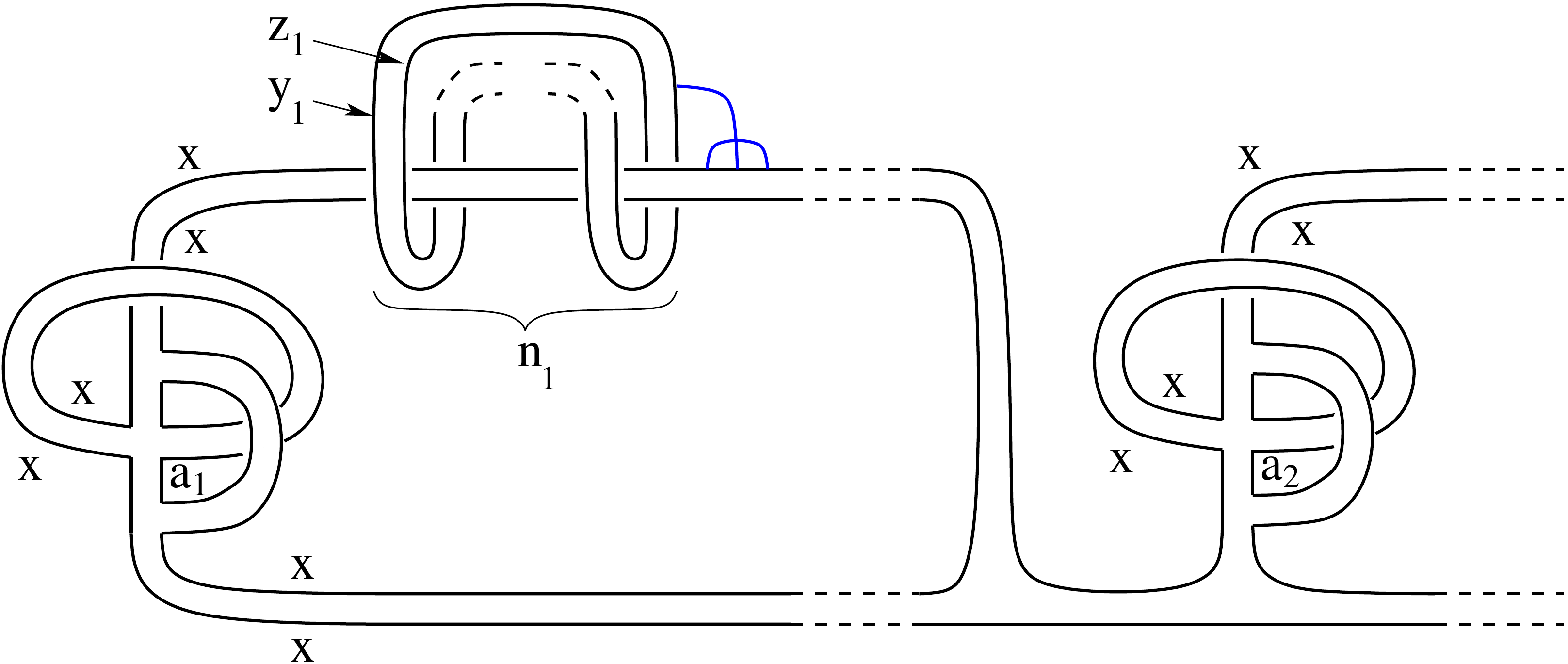}
	\end{center}
	\caption{Constructing relators over base surface}
	\label{realization1}
\end{figure}

\begin{theorem}\label{thm:realization}
	Let $G$ be a finitely presented group. 
	Then there exists a surface ribbon $S$ such that $ h(S) \cong F_k* G$ 	for some positive integer $k$. 
\end{theorem}
\begin{proof}
	First observe that if $G$ factors as free product of subgroups $G = G_1* G_2* \cdots *G_q$ then, using Lemma~\ref{lem:disjoint}, we can reduce the problem to finding ribbon surfaces whose reduced heap is isomorphic to the group heap of each factor $G_i$, since their disjoint union 
	(split sum)
	would realize $G$. 	We may assume that $G$ is irreducible with respect to free product factorization. Let $G = \langle \, \alpha_1, \ldots , \alpha_n \mid  R_1, \ldots , R_m \, \rangle$. We construct the required surface $S$ in various steps, at each of which we consider the effect on the reduced heap. We start by introducing a surface ribbon $S$ consisting of $m$ single trivial handles concatenated horizontally. We  realize each relator of $G$ on the handles of  this surface, which we will refer to as ``base surface''. Close to the left foot of each handle we apply the stabilization procedure as in Figure~\ref{stablabel}. As a consequence we have two free generators, for each stabilization, indicated by the letters $x$ and $a_i$ as in Figure~\ref{stablabel}, for $h(S)$. This is depicted in Figure~\ref{realization1} for the first two handles. Also, the labels of the handle of $S$ are $x$ for both edges of the handle, outside the small region where 
	 $a_i$
	is located, i.e. at the left foot of each handle of the base surface. Let us now consider the first relator $R_1$ in the presentation of $G$. Let $w(\alpha_1,\ldots , \alpha_n) = \alpha_{i_1}^{n_1}\cdots \alpha_{i_r}^{n_r}$ be the word corresponding to $R_1$. 
	We may assume that $w(\alpha_1,\ldots , \alpha_n) $ is reduced and, 	we suppose that $n_1$ is positive. 
			 We introduce an
	annular surface ribbon  component 
	wrapping 	$n_1$ times around the handle of $S$, as shown on the upper part of the first handle of Figure~\ref{realization1}. 
		For a negative  $n_1$, we take negative twists for the wrapping ribbon.
Let us denote by $y_1$ and $z_1$ 
the outer and inner arcs, respectively, of the annular ribbon 
that has been introduced. Then, depending on the 
positive or negative 
crossing, respectively, 
 $x$ changes to $x(y_1^{-1}z_1)^{n_1}$ or $x(z_1^{-1}y_1)^{n_1} = x(y_1^{-1}z_1)^{-n_1}$. We assume that it is the first case, without loss of generality.
We observe that since the handle of $S$ has undergone stabilization, as in the proof of Proposition~\ref{prop:stab}, the two arcs delimiting the handle of $S$ have the same color $x$ and do not modify the colors $y_1$ and $y_2$ after overpassing the annular ribbon that has been introduced, so no relator is derived on this ribbon.  This is also crucial in the fact that the color $x$ changes to $x(y_1^{-1}z_1)^{n_1}$, since if the base handle were not monochromatic, 
  its colors would have obstructed us from obtaining a multiplication by a simple power $(y_1^{-1}z_1)^{n_1}$. We do not perfom any chagnes on the other handles of the base surface. Let us denote the surface just obtained as $S_1$.  Proceeding as in the proof of Theorem~\ref{thm:heap}, we obtain a presentation of $h(S_1)$ with $m+2$ generators $x,a_i, y_1, \alpha_1$, where we have set $\beta_1 := y_1^{-1} z_1$, 
  and a single relator $\beta_1^{n_1}$   because the band after wrapping ribbon connects to the base labeled $x$. 
  We repeat the same construction of adding another annular ribbons    that links the handle of $S$, but does not link other parts. 
   Let us denote by $y_2$
  and $z_2$ the outer and inner arcs, respectively, of the newly introduced handle. 
    We assume, as before, that the crossing with the handle of $S$ introduces a new color $\beta_2:= y_2^{-1}z_2$. We let this handle wrap around the base handle $n_2$ times and, moreover, we use the same convention on signs as before. 
  We connect these annular ribbons to the base handle labeled $x$ by stabilization, 
  as indicated by a small blue arcs at the right of the annular ribbon in Figure~\ref{realization1}.   
  This addition introduces a relation $x=y_i$ and addition of a free generator corresponding to a small loop ($y$ in Figure~\ref{stab} (C)). 
   Let us denote by $S_2$ the surface ribbon obtained via this procedure. The fundamental heap of $S_2$ has $m+3$ free generators 
  (labeled $a_i$'s) and one relator $\beta_1^{n_1}\beta_2^{n_2}$.  
	
\begin{figure}[htb]
	\begin{center}
		\includegraphics[width=3in]{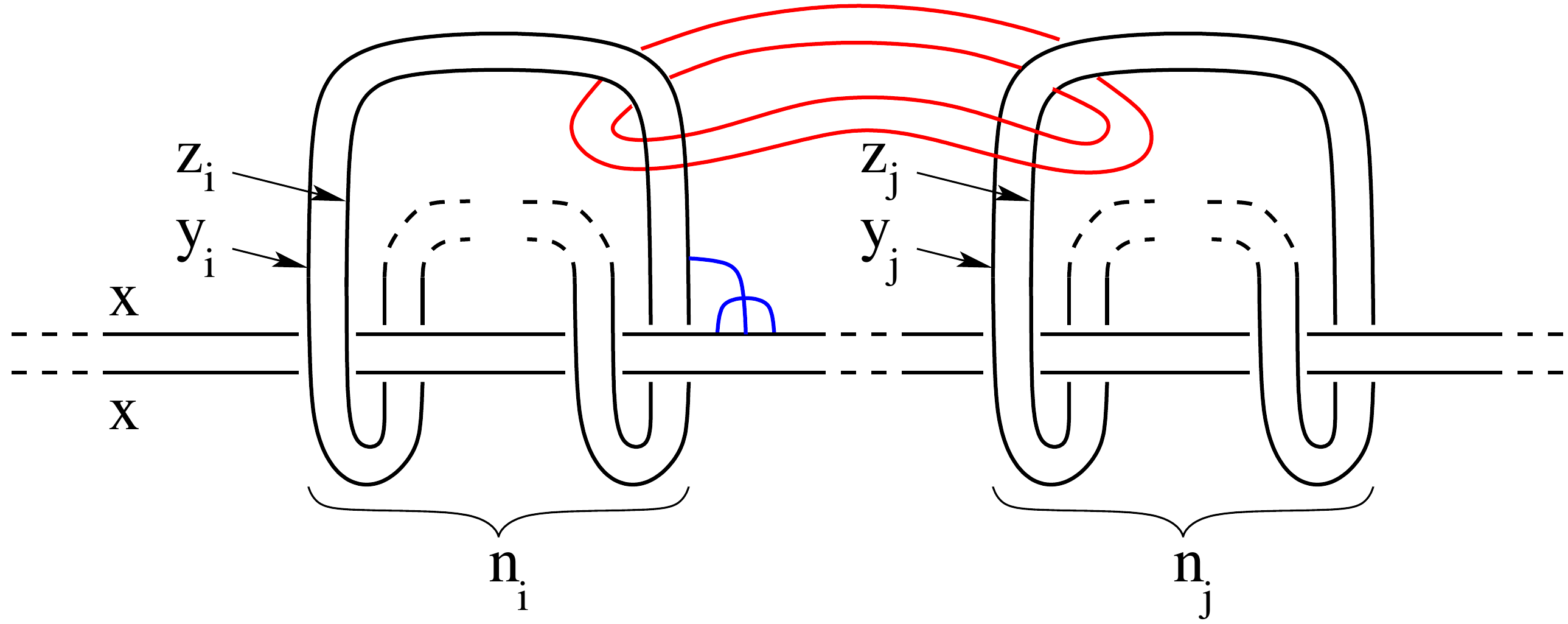}
	\end{center}
	\caption{Identifying generators through stabilization}
	\label{realization2}
\end{figure}

	 We proceed in this way $r$ times to construct a surface ribbon which we call $S_r$, until we have a relator $\beta_1^{n_1} \cdots \beta_r^{n_r}$. The word $w(\alpha_1, \ldots , \alpha_n)$ might have repetitions of $\alpha_i$'s in the product $\alpha_{i_1}^{n_1}\cdots \alpha_{i_r}^{n_r}$. For each repeated pair $\alpha_{i_t} = \alpha_{i_s}$, we 
		 add bands 
	  to connect the respective connected components of $\beta_{i_t}$ and $\beta_{i_s}$, being careful to let the connecting band overpass any other ribbon along the way, in such a way that $\beta_{i_t} = \beta_{i_s}$. This is diagrammatically indicated in Figure~\ref{realization2}. Since the overpassing ribbons are monochromatic, no new relator is introduced, so that the presentation of the fundamental heap $h(S_r)$ has a single relator that is not 
	  changed 
	  except for the effect of having certain letters $\beta_i$ identified, mirroring the same 
	  word 	  of $R_1$.
	
	At the end of this procedure we have found a surface $S_r$ whose fundamental heap is the free 
	product 
	of free factors (determined in number by the stabilizations performed) and a single relator obtained from $w(\alpha_1,\ldots, \alpha_r)$ upon substituting $\beta_1$ to $\alpha_{i_1}$,  $\beta_2$ to $\alpha_{i_2}$ etc. 
	
	Next, we 
	construct a ribbon surface on the second handle of the base surface, in the same way as in the previous step, that introduces a relator $\gamma_1^{m_1}\cdots \gamma_\ell^{m_\ell}$ 
	 that coincides with $R_2$ upon substituting the  $\gamma$'s in place of appropriate $\alpha_i$'s in the reduced word representing the relator $R_2$. To complete this step, we need to add untwisted bands between the handles just introduced and those in the previous step, in order to equate the generators $\beta_i$ and $\gamma_j$ that are substituted to the same generator $\alpha_k$ of $G$. Letting the band that is introduced in the process overpass every ribbon met along the way, at each step, we do not introduce new relators. This is similar to Figure~\ref{realization2}, where instead of introducing a band between surface ribbon over the same base handle, we connect surface ribbons over different base handles.
	 We also note that the other boundary curves do not create additional relations, as they run parallel to the outer boundary curve, and have the same letter $x$ assigned, producing the same relator words $R_i$ from each factor of boundary connected sum. 
	
	After performing the previous steps $m$ times, we obtain a surface ribbon $\tilde S$ whose fundamental heap consists of free products of 
	a number of 
	 free factors, and relators $R'_1, \ldots , R'_m$ where $R_i$ and $R'_i$ correspond to the same words, and differ only by appropriate changes of variables from $\alpha_j$ to $\beta_k$, according to some correspondence determined during the construction of $\tilde S$. 
	
	Finally, there is a mapping from $G$ onto 
	$\bar G := \langle\,  \beta_1, \ldots, \beta_n \mid R'_1, \ldots, R'_m \, \rangle$ determined by the assignment $\alpha_i \mapsto \beta_i$. This is clearly well defined since by construction we have that $R'_j$ is obtained from $R_j$ upon substituting the letters $\beta_i$ to $\alpha_i$. This gives an isomorphism between $G$ and $\bar{G}$. Since we have that $h(\tilde S) = F_k* \bar{G}$ for some $k$ given by $\#\{{\rm Stabilizations\ performed}\} + 1$, this completes the proof.
\end{proof}

\begin{remark}
{\rm	
	 We point out that the number of free factors $F_k$ appearing in the constructive proof of Theorem~\ref{thm:realization} depends on the particular presentation of the group $G$ that is chosen at the beginning. Given two different presentations of $G$, we obtain two generally different surface ribbons $S_1$ and $S_2$ whose fundamental heap is isomorphic to the group heap of $G$, up to a number of free factors. It is, therefore, desirable to determine a relation between the number of free factors appearing in the previous construction in terms of a given presentation of a group $G$. 

Let 
		$\langle \, \alpha_1, \ldots, \alpha_n\ |\ R_1, \ldots, R_m \, \rangle$ denote a presentation of $G$. Let $S$ denote the surface ribbon constructed in Theorem~\ref{thm:realization}, where $k$ is the number of free factors appearing in $h(S)$. Then we have $k = n + m +1$.
To see this, first note that in the construction of $S$ in Theorem~\ref{thm:realization}, each handle of the base surface is stabilized once, and contributes a free factor corresponding to a small loop
denoted by $a_i$ in Figure~\ref{realization1}. These handles correspond to relators, so that 
the number of these free factors is $m$. 
For the first generator $\alpha_{i_1}^{n_1}$ appearing in the first relator $R_1$: $w(\alpha_1,\ldots , \alpha_n) = \alpha_{i_1}^{n_1}\cdots \alpha_{i_r}^{n_r}$ in the proof of Theorem~\ref{thm:realization},
the handle wrapping around the first base handle depicted in Figure~\ref{realization1}
has a stabilization connected to the base handle depicted by a small blue arc in the figure, 
which contributes one free factor. 
When the same generator appears again, the corresponding wrapping handle is connected to the first as depicted in Figure~\ref{realization2}, and is not stabilized, so that it does not contribute any additional  free factor. Therefore each generator contributes one free factor. 
The external boundary component labeled $x$ runs over all handles and contributes one free factor.
 Hence we obtain  $k = n + m +1$.

}
\end {remark}

\section{Reversibility and additivity conditions}\label{sec:RAconditions} 

In this section, we consider two algebraic conditions for TSD operations and corresponding 2-cocycle conditions. Such additional conditions on 2-cocycles are used in the next section for constructions of cocycle invariants for surface ribbons.

\subsection{Reversibility and additivity for TSD operations}

\begin{definition}
{\rm
Let $(X, T)$ be a set with a ternary operation $T$.
We say that $(X, T)$ satisfies the {\it idempotency}  condition if 
$  T ( w,x,x) = w $ for all $w,x \in X$.
We say that $(X, T)$ satisfies  the {\it reversibility}  condition if 
$  T ( T(  w, x, y ), y, x) = w $ for all $w,x,y \in X$.
We say that $(X, T)$ satisfies  the {\it additivity}  condition if 
$  T ( T(  w, x, y ), y, z) = T(w, x, z) $ for all $w,x,y,z \in X$.
}
\end{definition}

\begin{remark}
{\rm
Additivity and idempotency conditions imply the reversibility condition by setting $x=z$.
}
\end{remark}

\begin{remark}
{\rm
The reversibility and additivity conditions are used in Lemmas~\ref{lem:fund} and \ref{lem:color} 
for the well-definedness of the fundamental heap and colorings under the YI and IY moves.
}
\end{remark}

\begin{example}
{\rm
Direct computations show that any group heap satisfies 
reversibility and additivity conditions.
}
\end{example}

\begin{example}
{\rm
Let $X$ be a module over $\Z [ t^{\pm 1}, s]$ and define a ternary operation by $T(x,y,z)=tx + sy + (1-t-s)z$. 
It is computed that $T$ is a TSD operation (e.g., \cite{ESZcascade}). 
Then direct calculations show that $X$ does not satisfy the reversibility and additivity conditions in general.
Thus these conditions can be used to detect non-heap TSD operations.
}
\end{example}

\subsection{Reversibility and additivity for TSD 2-cocycles}

\begin{definition}\label{def:cocy}
{\rm
Let $X$ be a group heap, and let $A$ be an abelian group.
A 2-cocycle $\psi \in Z^2_{\rm SD} (X, A)$ is said to be  {\it nondegenerate}
if $\psi(x,y,y)=0$ for all $x, y \in X$ \cite{ESZheap}. 
A 2-cocycle $\psi \in Z^2_{\rm SD} (X, A)$ is said to satisfy  the {\it reversibility} condition 
if it holds that 
$$ \psi(w, x, y) + \psi (w x^{-1} y, y, x ) = 0 $$
for all $w,x,y \in X$. 
A 2-cocycle $\psi$ 
 is said to satisfy  
the {\it additivity} condition if 
it holds that 
$$ \psi (w, x, y) + \psi(w x^{-1}y, y, z) = \psi (w, x, z) $$
for all $w,x,y,z \in X$. 
}
\end{definition}

\begin{remark}\label{rmk:additive}
{\rm
If a 2-cocycle $\psi\in Z^2_{\rm SD}(X, A)$ is nondegenerate, then the additivity condition on $\psi$ implies the reversibility condition.
}
\end{remark}

The reversibility and additivity conditions  ensure well-definedness 
of the cocycle invariant defined in the next section.

Direct computations imply the following.

\begin{lemma}\label{lem:revadd}
Let $X$ be a group heap and $A$  an abelian group.
Then any 2-coboundary $\delta f$, $f \in Z^1_{\rm SD} (X, A)$, satisfies the reversibility and additivity 
conditions. Furthermore, linear combinations of reversible and additive 2-cocycles are reversible and additive, respectively.
\end{lemma}

In \cite{CS}, it was shown that certain identities satisfied by a quandle induce subcomplexes.
Similarly, it is expected that Lemma~\ref{lem:revadd} extends to higher dimensions to form
corresponding subcomplexes. 
We pose the following definition for dimension $2$. 

\begin{definition}
	{\rm 
	\begin{sloppypar}
Let $X$ denote a heap and let $H^2_{\rm SD}(X,A)$ denote the self-distributive second cohomology group of $X$. Let $Z^2_{\rm RA}(X,A)$ denote the subgroup of $Z^2_{\rm SD}(X,A)$ of $2$-cocycles that are reversible and additive. 
As an application of Lemma~\ref{lem:revadd}, the quotient $H^2_{\rm RA}(X,A) := Z^2_{\rm RA}(X,A)/B^2_{\rm SD}(X,A)$ is a well defined 
subgroup 
of $H^2_{\rm SD}(X,A)$, which we call the {\it reversible and additive} second cohomology group of $X$, or {\it RA cohomology} for short, and similarly, the RA cocycle group for $Z^2_{\rm RA}(X,A)$.
\end{sloppypar}
}
\end{definition}

Direct computations show the following.

\begin{proposition}
Let $(X,T)$  be a TSD set satisfying  reversibility and additivity conditions, $A$  an abelian  group, and 
let  $\psi: X^3 \rightarrow A$ be a function. 
Define a ternary operation on $X \times A$ by 
$ \hat{T} ((x, a), (y, b), (z, c) ) := ( T( x, y, z) , a + \psi (x, y,z) )$.  Then, $\hat T$ defines a TSD structure on $X\times A$ if and only if $\psi$ is a TSD $2$-cocycle. Moreover,
 $\hat{T}$ satisfies reversibility and additivity condition if and only if
$\psi$ satisfies the reversibility and additivity conditions.
\end{proposition}

In order to obtain 2-cocycles with reversibility and additivity conditions, we review constructions of cocycles from \cite{SZframedlinks}.
For completeness we include a proof of the first construction, while defer the reader to \cite{SZframedlinks} for a proof of the second.

\begin{lemma}\label{lem:Zn}{\rm (\cite{SZframedlinks})}
Let $\Z_n=\langle \, \zeta \mid \zeta^n=1 \, \rangle$ be the cyclic group of order $n$ 
in multiplicative notation with a generator $\zeta$.
Let $\phi_i=\sum_{x \in \Z_n } [ \sum_{j=0}^{n-1}  \chi_{(x,\zeta^j,\zeta^{j+i})} ] $, $i=1, \ldots, n-1$,
where $\chi$ denotes the characteristic function. 
Then $\phi_i$ is a nondegenerate $2$-cocycle, $\phi_i \in C^2_{\rm NDH}(\Z_n, \Z)$, for all 
$i=1, \ldots, n-1$.
Moreover, reducing coefficients modulo $n$, we obtain $\phi_i \in C^2_{\rm NDH}(\Z_n, \Z_n)$ for all $i$.
\end{lemma}

\begin{proof}
For a fixed $i$, the 2-cocycle $\phi_i$ vanishes for 2-chains $(x, u, v ) \in C_2^{\rm NDH} (X, \Z)$ if 
$v \neq u r^i$. 
Hence if $v\neq u r^i$, then the last two terms of $(*)$, $- \phi(x, u, v) +  \phi(x y^{-1} z , u, v) $,  both vanish.
If $v=ur^i$, then  both terms are $1$ and cancel. Hence we focus on the first two terms.

Let $v=u\zeta^k$, then the first two terms of $(*)$ are 
$\phi(x, \zeta^j, \zeta^{j+m}) - \phi(x \zeta^{k}, \zeta^{j+k}, \zeta^{j+m+k})$ for some $j, m \in \Z_n$. 
If $m=i$, then both terms are 1 and cancel. 
If $m\neq i$, then  both vanish. 
%
Hence $(*)$ holds.
\end{proof}

\begin{lemma} \label{lem:Dn}{\rm (\cite{SZframedlinks}) }
Let $D_n$ be the dihedral  group of order $2n$ generated by a rotation $\zeta$ and reflection $a$ with a relation $ a\zeta a = \zeta^{-1}$. 
Let $\psi_i=\sum_{x \in D_n } [ \sum_{j=0}^{n-1} (  \chi_{(x,\zeta^j,\zeta^{j+i})} +\chi_{ (x, a\zeta^{-j},a\zeta^{-j-i}  ) } )  ] $, $i=1, \ldots, n-1$.
Then $\psi_i$ is a nondegenerate $2$-cocycle, $\psi_i \in C^2_{\rm NDH}(D_n, \Z)$, for all 
$i=1, \ldots, n-1$. 
Moreover, reducing coefficients modulo $n$, we obtain $\psi_i \in C^2_{\rm NDH}(\Z_n, \Z_n)$ for all $i$.
\end{lemma}


%
%
%

\begin{lemma}\label{lem:cocy}
Let $\vec{a} = (a_0, a_1, \ldots , a_{n-1} ) \in (\Z_n)^{n}$, where we set $a_0=0$.
Let $\phi_{\vec{a}}=\sum_{i=0}^{n-1} a_i \phi_i$ and  $\psi_{\vec{a}}=\sum_{i=0}^{n-1} a_i \psi_i$, where $\phi_i$  and $\psi_i$ are cocycles in Lemma~\ref{lem:Zn} and \ref{lem:Dn} with coefficients modulo $n$, respectively,
and we formally set $\phi_0=\psi_0\equiv 0$. 
Then $\phi_{\vec{a}}$ and $\psi_{\vec{a}}$ satisfy  the reversibility and additivity conditions if and only if 
$a_{k+\ell}=a_k + a_\ell$ for all $k, \ell$, where $k, \ell$ are taken  modulo $n$. 
\end{lemma}

\begin{proof}
For $\phi$, any triplet $(x,y,z)\in (\Z_n)^3$ can be written as $(x, \zeta^j, \zeta^{j+i})$ for some $i, j \in \Z_n$.
Since $\phi(x, \zeta^{j}, \zeta^{j+i})=a_i$, 
one computes that the additivity 
$$\phi(x, \zeta^{j}, \zeta^{j+k}) + \phi ( x\zeta^j, \zeta^{j+k}, \zeta^{j+k+\ell} ) = \phi(x, \zeta^{j}, \zeta^{j+k+\ell}) $$ is 
equivalent to  $a_{k} + a_\ell = a_{k + \ell}$, where the subscripts are considered modulo $n$. The invertibility 
$\phi(x, \zeta^j , \zeta^{j+k}) = -  \phi( x \zeta^\ell , \zeta^{j+k}, \zeta^j )$ is equivalent to $a_{-k}=- a_k$, which 
follows from  the equation $a_{k} + a_\ell = a_{k + \ell}$.
Similar arguments apply to $\psi$. 
\end{proof}

\begin{remark}
{\rm
The preceding lemma implies that a cocycle for $\Z_n$ and $D_n$ are determined by the value of 
$a_1 \in \Z_n$. It is proved in \cite{SZframedlinks} that $\phi$ and $\psi$ are non-trivial in $H^2_{\rm SD} (X) $ 
 and generally linearly independent. From Examples~5.12 and 5.13 in that article, one also sees that if $\vec{a}\neq 0$, the cocycles $\phi_{\vec{a}}$ and $\psi_{\vec{a}}$ are nontrivial as well. 
}
\end{remark}

		Let $R$ be a ring considered  with abelian heap operation with respect to its additive structure.
		 In \cite{SZframedlinks} it was shown that  the $2$-cochains $\psi_{(a,b,c)} (x,y,z) = (ax+ b(z-y) + c)(z-y)$ are nontrivial $2$-cocycles for any choice of $a,b,c\in R$ with $a \neq 0$. In fact, it turns out that $\psi_{(a,b,c)} - \psi_{(a,b,0)}$ is a coboundary 		 and, therefore, we will omit the index $c$ in the rest of the article, and simply write $\psi_{(a,b)}$ for $\psi_{(a,b,0)}$. 
		
		\begin{lemma}\label{lem:cocyb}
			The cocycles $\psi_{(a,b)}$ are reversible and additive if and  only if $a = 2b$. 
		\end{lemma}
	\begin{proof}
		Since each $\psi_{(a,b)}$ is non-degenerate, it is enough to show that $\psi_{(a,b)}$ is additive. The additive condition 
		$$
				\psi_{(a,b)}(w,x,y) + \psi_{(a,b)}(w-x+y,y,z) = \psi_{(a,b)}(w,x,z)
		$$
		becomes 
		$$
				(aw+b(y-x))(y-x) + (a(w-x+y)+b(z-y))(z-y) = (aw+b(z-x))(z-x).
		$$
		This is readily seen to hold for all $w,x,y,z$ if and only if $a = 2b$.
	\end{proof}
We set $\psi_b := \psi_{(2b,b)}$ to denote the reversible and additive cocycles of Lemma~\ref{lem:cocyb}.

\subsection{Mutually distributive RA  cocycles}

Let $X$ be a group heap and let $A$ be an abelian group. Suppose that $\psi_1$ and $\psi_2$ are reversible and additive $2$-cocycles. 
If the pair $(\psi_1,\psi_2)$ is mutually distributive
 (as defined in Section~\ref{sec:TSDcoh}),
then we say that $(\psi_1,\psi_2)$ is a {\it reversible and additive mutually distributive pair}, or RA mutually distributive pair for short.

\begin{example}
	{\rm 
Let $\psi_b$ and $\psi_d$ denote two RA $2$-cocycles of Lemma~\ref{lem:cocyb}. Then in order to verify whether they are mutually distributive, it is enough, by symmetry,  to show the equality 
$$
\psi_b(x,y,z) + \psi_d(x-y+z,u,v) = \psi_b(x-u+v,y-u+v,z-u+v) + \psi_d(x,u,v).
$$
Using the definition of $\psi_b$ and $\psi_d$ we see that this holds for all $x,y,z,u,v$ if and only if $2(b-d) = 0$. 
}
\end{example}

\begin{definition}\label{def:sep}
	{\rm 
Let $X$ be a group heap and let $A$  an abelian group. A $2$-cocycle $\psi$ is said to be {\it separable} if it satisfies the $2$-cocycle condition pairwise, as follows: 
\begin{eqnarray*}
\psi(xy^{-1}y,u,v) &=& \psi(x,u,v),\\
\psi(x,y,z) &=& \psi(xu^{-1}v,yu^{-1}v,zu^{-1}v)
\end{eqnarray*}
for all $x,y,u,v \in X$.
In other words, $\psi$ is separable if and only if $\psi$ and the zero cocycle are 
mutually distributive.
}
\end{definition}

\begin{example}\label{ex:RAcompatibleb}
	{\rm 
	The cocycles $\phi_i$ of Lemma~\ref{lem:Zn} are separable, as a direct computation shows. Since $\phi = \sum_{i= 1}^{n-1} i \phi_i$ is additive and reversible as well, it follows that $\phi$ and the zero cocycle are RA mutually distributive cocycles. Similarly, the cocycles $\psi_i$ of Lemma~\ref{lem:Dn} are separable and, therefore, $\psi = \sum_{i= 1}^{n-1} i \psi_i$ is RA mutually distributive with the zero cocycle. 
}
\end{example}

\begin{remark}
{\rm
If $\bar \psi =(\psi_i)_{i=1}^n$ is a sequence of separable cocycles, 
then they are pairwise mutually distributive. 
}
\end{remark}

\section{Colorings and cocycle invariants of ribbon graphs}\label{sec:cocyinvariant}

A coloring of a surface ribbon diagram by a heap  is defined by assigning elements of the heap to double arcs as follows, in a manner similar to quandle coloring, and cocycle invariants are also similarly defined as in~\cite{CJKLS}. In this section we give such definitions, 
realizations of surface ribbons with non-trivial invariant values, and an application to non-trivial cohomology.

\subsection{Colorings} 

First we define and examine colorings of surface ribbon diagrams by heaps.

\begin{definition}
{\rm
Let $X$ be a heap. 
Let $D$ be a surface  ribbon diagram and ${\cal A}$  the set of doubled arcs.
A {\it coloring } of $D$ by $X$ is a map ${\cal C}: {\cal A} \rightarrow X$ that satisfies the {\it coloring condition}
as depicted in Figure~\ref{buildingblocks}  (A) and (C), where $(z,w) = (xu^{-1}v, y u^{-1} v )$.
}
\end{definition}

From the definition we obtain the following by checking the moves. 
The proof parallels that of Lemma~\ref{lem:fund}.

\begin{lemma}\label{lem:color}
The sets of colorings of  two surface ribbon diagrams 
are in bijection between each move listed in Figure~\ref{moves}.
\end{lemma}

In particular, the number of colorings of a surface ribbon   
diagram by a finite heap $X$ is an invariant of a surface ribbon, 
that does not depend on the choice of a diagram, and is denoted by ${\rm Col}_X(S)$. 
Similarly to \cite{CJKLS,SZframedlinks}, the set of colorings of a surface ribbon $S$ by a heap $X$ can be considered as the set of heap homomorphisms
from $h(S)$ to $X$. Although the fundamental heap was defined by group presentations, these homomorphisms need not be group homomorphism; assigning a single color to all arcs that is not the identity element is a heap homomorphism but not a group homomorphism.

We observe that from the definition, if $x=y$ at a crossing as in Figure~\ref{buildingblocks} (A), then 
we have $z=w$. Consequently, if $x=y$ (the two colors are equal) at 
one pair of arcs of a ribbon, then the entire ribbon (band) has
this property. In this situation we say that this is a {\it monochromatic ribbon}.
We also note that if the over-arc is a
monochromatic ribbon, then 
the colors of the under-arc in Figure~\ref{buildingblocks} (A) 
satisfy $x=z$ and $y=w$, i.e.
a monochromatic overpassing ribbon does not change the colors of the under-arc.

\begin{example}\label{ex:colortrivial}
	{\rm 
	Let $S = B_1^n\natural B_2^m$ be the surface obtained by concatenating $n$ trivial handles and $m$ crossed handle pairs, as in Example~\ref{ex:loopband}. Then no 
	restriction on the assignment of colors arises for any 
	$X$.
	 The number of colors of $S$ by $X$ is $|X|^{n+m+1}$. 
}
\end{example}
 
 \begin{example}\label{ex:colortwisted}
 	{\rm 
 		Let $S = D(  m_1,\ldots , m_n)$ denote the surface of Example~\ref{ex:loops}, where we take the connected sum of $n$ copies of looped ribbons, 	with 
	$m_j$ twists for $j=1, \ldots, n$,
	constructed by attaching 
	copies of
	Figure~\ref{loops1} vertically.
	The $j^{\rm th}$ looped ribbon is denoted by $H^{m_j}$.
		We order the boundary components by taking $b_0$ to be the base of $S$
	(the components that contains the top through bottom outside curve in standard position), 
	and $b_1,\ldots , b_{n}$ are the inner boundary components of the handles in the order they appear in the standard position, from top to bottom. 
The colorings of $S$ by the cyclic group $X=\Z_m$ (taken here in multiplicative notation with a generator $\zeta$  as in Lemma~\ref{lem:Zn}) are determined as follows.

	The ribbon  $H^{m_j}$ has two boundary curves, the outer component belonging to
	$b_0$ and the inner component $b_{j}$. 
	The colors assigned are $y$ for the outer component (base) $b_0$ (that corresponds to
	the top and bottom arcs labeled by  $y$ in Figure~\ref{loops1}),
	and $x_j$  for the inner component $b_j$, that corresponds to $x$ in the figure, 
	where $j$ represents that $b_j$ is at the $j^{\rm th}$ handle  $H^{m_j}$. 
		Let $\alpha_j:= x_j^{-1} y$. 
	Then the coloring condition for $H^{m_j}$  is $\alpha_j^{m_j} =1$ from 
	Example~\ref{ex:loops}.
	 If $y = \zeta^t$ and $x_j = \zeta^{t_j}$, then $\alpha_j^{m_j} =1$ is equivalent to
	 $\zeta^{m_j (t-t_j)} =1$, that is, $m_j ( t-t_j ) \equiv 0$ modulo $m$.
	 Set $d_j := {\rm gcd}(m,m_j)$.
		 For each arbitrary $y$, there are $d_i$ solutions to the equation for $x_i$, namely given by $t - t_j = s_j \frac{m}{d_j}$, for $s_j = 0, \ldots, d_j-1$. The total number of colorings is $|X|\sum_j d_j$. 
 	}
 \end{example}

\subsection{Cocycle invariants with respect to the boundary curves} 

In this subsection  we consider heap cocycle invariants of surface ribbons.
Let $X$ be a heap and $D$  a surface ribbon diagram.
Then a 2-cocycle invariant is defined in a manner similar to the quandle 2-cocycle invariant
as follows.

\begin{figure}[htb]
\begin{center}
\includegraphics[width=1.3in]{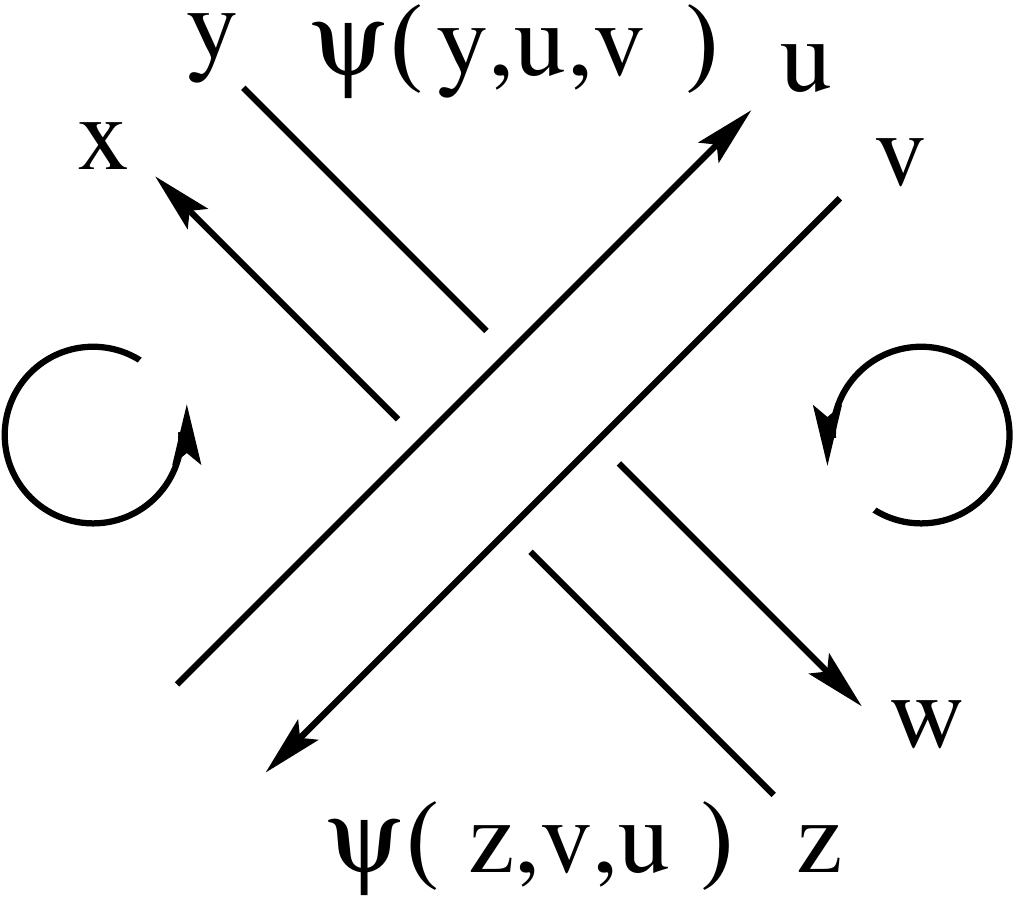}
\end{center}
\caption{Orientations and labeling across double arcs}
\label{ori}
\end{figure}

The orientations of boundary curves of a surface ribbon diagram $D$ are defined as depicted in 
Figure~\ref{ori}. The curves of a ribbon are oriented in antiparallel directions, in such a way that 
are consistent with the counterclockwise orientation of the complementary regions as depicted.

\begin{definition}\label{def:cocyinvariant}
{\rm
Let $X$ be a heap and  $A$  an abelian group 
with multiplicative notation.
Let $\bar \psi = (  \psi_1, \ldots, \psi_n ) $ be a sequence of pairwise mutually distributive RA $2$-cocycles 
of $X$ with coefficient group $A$,  so that $\psi_i\in Z^2_{\rm RA}(X,A)$ for all $i=1, \ldots, n$ and 
each pair is mutually distributive.
The  {\it 2-cocycle heap invariant} of a surface ribbon $S$ with respect to 
$\bar \psi = (\psi_1, \ldots, \psi_n)$ is defined as follows.

Let $D$ be an oriented diagram of a surface ribbon $S$.
Let $S= S_1\cup \cdots \cup S_n$ denote the connected components of $S$, and 
$D= D_1\cup \cdots \cup D_n$ be the corresponding oriented diagram.
We decorate $D_i$ with the $2$-cocycle $\psi_i$ for each $i$.

For each $D_i$, we order the connected components of its boundary;
let $b_i^j$ be the ordered boundary components of $D_i$.  We fix a base point in each $b_i^j$ and order the arcs of $D_i$ following the orientation of $b_i^j$, as well as the crossings  
where $b_i^j$ underpasses. 
Let $\tau_i^j(1), \ldots , \tau_i^j(r)$ denote the crossings where $b_i^j$ underpasses in this order. 

We define, for a given coloring $\mathcal C$, $\Psi_{\bar \psi}^i(\mathcal C, D) = \otimes_j \prod_k B^{\bar \psi}(\mathcal C,\tau_i^j(k))$, where the tensor product runs over each boundary component of $D_i$, the product runs over all the crossings where boundary components $b_i^j$ underpasses, and at each crossing we use the cocycle that decorates the overpassing connected component:
$B^{\bar \psi}(\mathcal C,\tau_i^j(k)) = \psi_{\ell(i,j,k)}  (x, y, z)$, where $\ell (i,j,k) $ is the number assigned to 
the overpass at $\tau_i^j (k)$, $x$ is the color assigned to  the undrpass arc right before $\tau_i^j (k)$, $(y,z)$ is 
the   pair of colors assigned to the overpass that appear in this order (cf. Figure~\ref{heaptypeIIIcocy}). 
The invariant values are considered equivalent up to permutations of tensor factors of each 
$\Psi_{\bar \psi}^i(\mathcal C,S)$, to allow renumbering boundary components.
Thus the value is regarded as an element of the symmetric algebra $S(\Z[A])$ (in fact its subspace 
of degree being the number of boundary components), but we also take a tensor form
$x_1 \otimes \cdots \otimes x_k$ as invariant values, regarding it as a representative of 
elements of $S(\Z[A])$, so that tensors of permutations of $x_j$s are considered equal. 
Then we set
$$
\Psi_{\bar \psi}(S) = \sum_{\mathcal C} (\Psi_{\bar \psi}^1(\mathcal C,S), \ldots , \Psi_{\bar \psi}^n(\mathcal C,S)),
$$
where each entry of this formal vector corresponds to one connected component of $S$, and the sum refers to each entry of the vector component-wise.   
}
\end{definition}

\begin{figure}[htb]
\begin{center}
\includegraphics[width=1.3in]{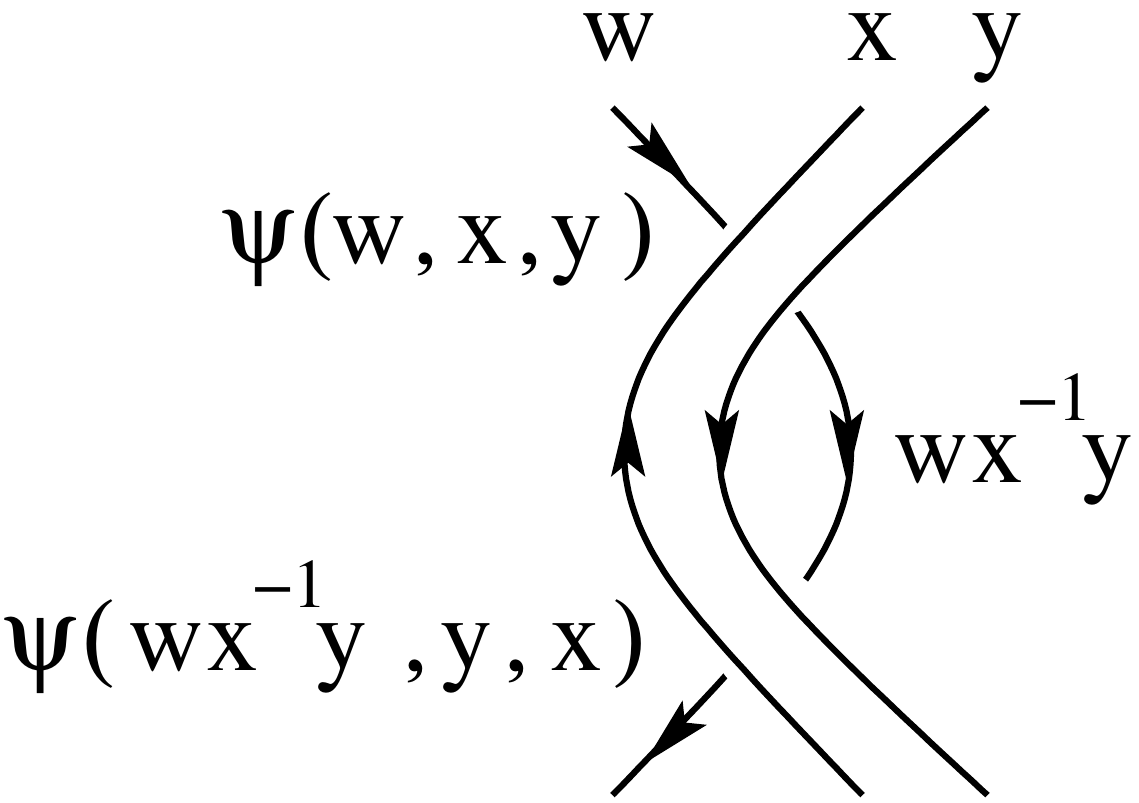}
\end{center}
\caption{Invariance under the Reidemeister type II move}
\label{typeII}
\end{figure}

\begin{figure}[htb]
\begin{center}
\includegraphics[width=3.5in]{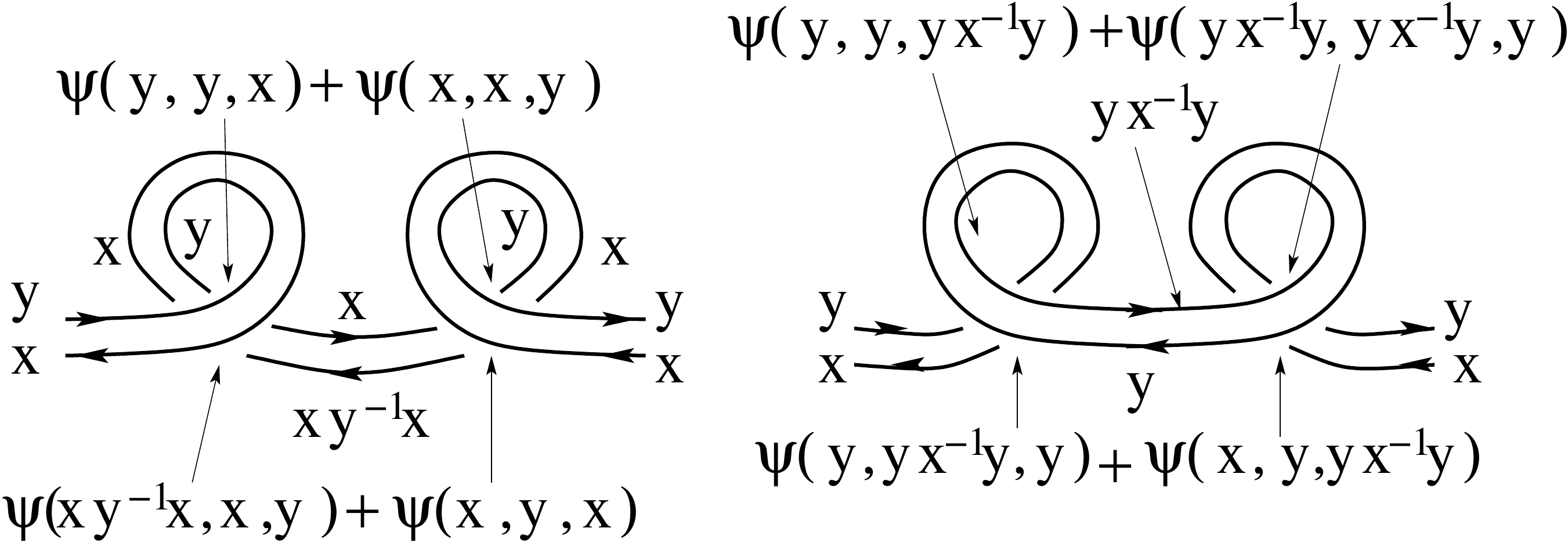}
\end{center}
\caption{Invariance under the cancelation move}
\label{CL}
\end{figure}

The following is a special case of labeled homology defined in \cite{ESZcascade}
restricted to RA cocycles.

\begin{definition}
{\rm
Let $X$ be a heap and  $A$ an abelian group 
with multiplicative notation.
Let $\bar \psi = (  \psi_1, \ldots, \psi_n ) $ be a sequence  of pairwise mutually distributive RA $2$-cocycles 
of $X$ with coefficient group $A$.

We call $\bar \psi$ a {\it coboundary} if  $\psi_i$ are coboundaries simultaneously, that is, 
 there is a 1-cochain $f \in C^1_{\rm SD}(X,A)$ such that 
$\psi_i=\delta f$ for all $i=1, \ldots, n$. 
Two sequences $\bar \psi = ( \psi_i)_{i=1}^n $ and $\bar \psi' = ( \psi_i')_{i=1}^n$
are called {\it cohomologous} if $\bar \psi - \bar \psi'$ is a coboundary. 
The equivalence classes by this relation of cohomologous is called 
the {\it cohomology class}  $[\bar \psi]$, and they form an abelian group by component addition. 
The group of cohomologous classes of $\bar \psi$ is
denoted by $\bar H^{2,n}_{RA}(X,A)$. 
}
\end{definition}

The following is proved by arguments similar to those found in \cite{CJKLS}, with the only difference that we need to 
 take into consideration that different cocycles may decorate different  overpassing connected components.

\begin{theorem}\label{thm:cocyinvariant}
	The $2$-cocycle heap invariant is indeed an invariant of surface ribbons. Moreover, a labeled 2-coboundary yields an integer multiple of the vector with trivial tensors in its entries (with an appropriate number of entries, and an appropriate number of tensor products in each entry). 
	The cocycle invariant  $\Psi_{\bar \psi} (S)$ depends only on the labeled cohomology class $[\bar \psi]\in \bar H^{2,n}_{\rm RA}(X,A)$. 
\end{theorem}
\begin{proof} 
The invariance is proved by checking Reidemeister moves.
More than one component of overpassing ribbons appears only in the type III move,
so that it is delayed to the last, and the remaining cases are checked for a single cocycle 
assigned on overpassing ribbons.
The invariance under the Reidemeister type II move (RII in Figure~\ref{moves}) 
follows from the reversibility condition of 2-cocycles as depicted in Figure~\ref{typeII}. Observe that this is done for a single under-arc and, therefore, it shows invariance of $\Psi_{\psi} (S)$ with respect to each boundary component.
The IH move does not involve cocycles and keeps $\Psi$ unchanged.
The cancelation move (CL in Figure~\ref{moves}) follows from the reversibility condition 
and the equality 
 $\psi(x, y,z)=\psi(x, z, z y^{-1} z ) $ 
which  is obtained by setting $y=u$ and $z=v$ in  the 2-cocycle condition
$(*)$  and changing  variables.
The cancelations of 2-cocycles under the CL moves are depicted in Figure~\ref{CL}, where
$+$ indicates pairs of terms that cancel. It is clear that the canceling terms are paired with respect to different arcs and, consequently, weights corresponding to different boundary components remain unchanged.
The invariance under YI move and IY move follow from the additivity condition and reversibility condition, respectively,
and depicted in Figures~\ref{Yori} and \ref{Yoriover}. Observe that Figure~\ref{Yori} refers to a single boundary component that is slid beneath a fat vertex, while Figure~\ref{Yoriover} is obtained from the single arc Reidemeister move II of Figure~\ref{typeII} relative to arcs $y$ and $u,v$.

\begin{figure}[htb]
\begin{center}
\includegraphics[width=1.8in]{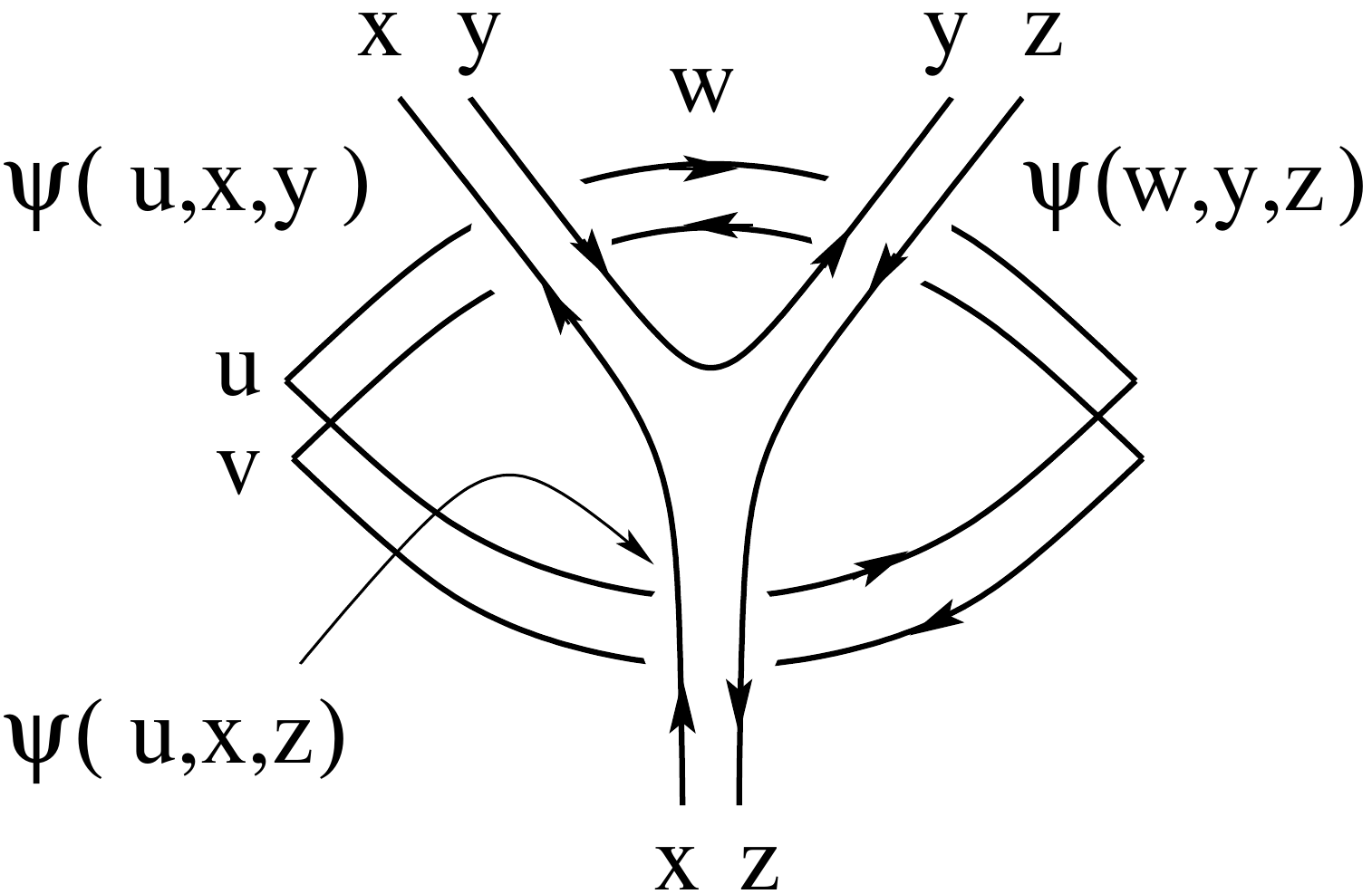}
\end{center}
\caption{Labels of arcs behind  a trivalent vertex}
\label{Yori}
\end{figure}

Lastly we check the type III move, refer to Figure~\ref{heaptypeIIIcocy}.
In the figure, there are two overpassing ribbons, the middle one $R_1$ labeled by $(x,y)$ at the top, 
and the top one  $R_2$ labeled by $(u,v)$.
Assume that
 $R_1$ and $R_2$
belong to distinct connected components of the given surface ribbon $S$,
and assigned two cocycles $\psi_1$ and $\psi_2$, respectively. 
Then the ribbon crossings are assigned cocycle values as indicated in the figure,
and the equality of the LHS and RHS is exactly the definition of mutual distributivity in Section~\ref{sec:TSDcoh}.
If $R_1$ and $R_2$ belong to the same component, then the equality follows from the original 2-cocycle condition.
It follows that $\Psi_\psi (S)$ is well defined.

Let $S$ be a surface ribbon diagram.
Then colored boundary components  represent 2-cycles of  $Z_2^{\rm SD} (X, \Z)$.
Let  $S_{\mathcal C}$ be a colored diagram, and $b_{\mathcal C}$ be one of the boundary components
of $S_{\mathcal C}$.
Let $\tau$ be a crossing of $b_{\mathcal C}$ where it goes under a ribbon colored by $(y,z)$ in this order,
changing the color from $x$ to $xy^{-1}z$. If the cocycle assigned to the overpassing ribbon is $\psi_i$,
then the weight assigned to $\tau$ is $\psi_i(x,y,z)$. 
Suppose $\bar \psi $ is a coboundary, then  $\psi_i=\delta f$ for some $f \in C^1_{\rm SD}(X,A)$,
so that  $\psi_i(x,y,z)= \delta f (x, y, z)=f(x) - f (xy^{-1}z)$. 
Assign $f(x)$ to the arc colored by $x$ near $\tau$, and $- f(x y^{-1} z) $ to the arc colored by $x y^{-1} z$.
Since $b_{\mathcal C}$ is a closed curve, these assigned values cancel at the both ends of each arc.
(This argument is similar to that of \cite{CJKLS}.)
Hence the tensor factor corresponding to $b_{\mathcal C}$ is trivial, $e \otimes \cdots \otimes e$, for the multiplicative identity $e$ of $A$. 
Then $\Psi_\psi (S)$ depends only on the cohomology class $[\bar \psi]$. 
In particular, if $\Psi_\psi (S)$ is non-trivial, then $[\bar \psi]\neq 0$.
\end{proof}

\begin{figure}[htb]
\begin{center}
\includegraphics[width=3.2in]{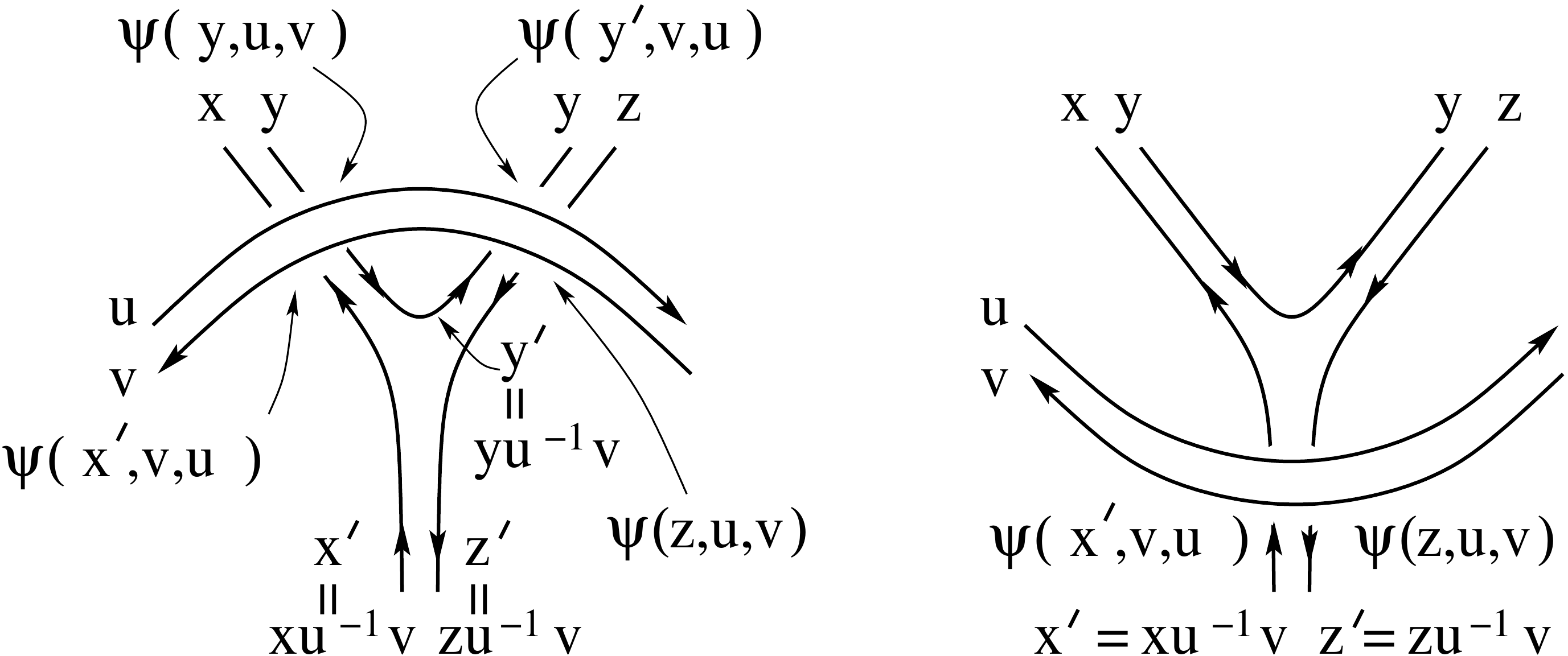}
\end{center}
\caption{Labels of arcs over  a trivalent vertex}
\label{Yoriover}
\end{figure}

\begin{remark}
	{\rm 
We observe that if $\psi = \delta f$, then the integer 
that is the coefficient of $e\otimes \cdots \otimes e$
 is the number of colorings of $S$ by $X$. This is, in fact, a direct consequence of the proof of the theorem, since for a given $\delta f$ each coloring determines a copy of the trivial vector.
}
\end{remark}

\begin{example}\label{ex:cocyloopband}
	{\rm 
	Let $S = B_1^n\natural B_2^m$ be the surface obtained by connecting $n$ single trivial handles and $m$ crossed handle pairs, as in Example~\ref{ex:loopband}. 
		There are $m+1$ components of  boundary curves.
	Let $X$ be a heap and let $\psi$ denote an additive and reversible $2$-cocycle with coefficients in $A$. Then, since each crossing contributes trivially to each boundary component, and the coloring conditions are trivially satisfied, 
 we have $\Psi_\psi(S) = |X|^{n+m+1}\cdot  e^{\otimes (m+1)}$, where $e$ is the unit of $A$.
 }
\end{example}

We consider an example where nontrivial contributions arise.

\begin{figure}[htb]
	\begin{center}
		\includegraphics[width=1.2in]{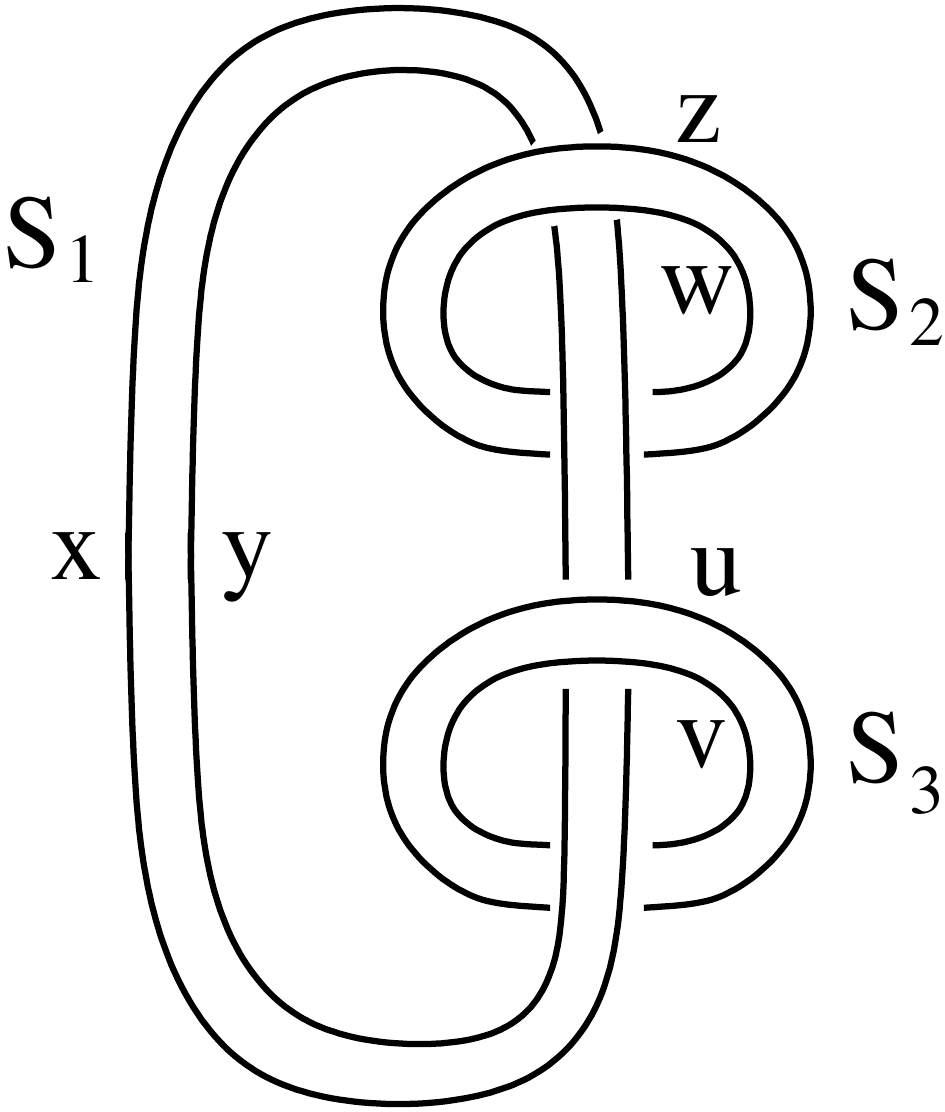}
	\end{center}
	\caption{Two unlinked annuli linking a third annulus}
	\label{rings}
\end{figure}

\begin{example}\label{ex:minimalnontrivial}
	{\rm 
	Let $S = S_1 \cup S_2 \cup S_3$ denote the surface consisting of three annuli where $S_2$ and $S_3$ link $S_1$ but are mutually unlinked, as depicted in  Figure~\ref{rings}. 
	Let $X = \mathbb Z_n$ in multiplicative notation generated by $\zeta$, $A = \mathbb Z_n$
	also  in 	multiplicative notation generated by $g$,
	and let $\phi$ denote the RA $2$-cocycle of Example~\ref{ex:RAcompatibleb} which is separable, and therefore mutually distributive with the zero cocycle. 
	Let us take the triple of  pairwise mutually distributive cocycles $(0,\phi,0)$, where $S_1$ and $S_3$ are decorated by
	the zero cocycles, and $\phi$ 
	decorates
	$S_2$. 
	Colors assigned to arcs are represented by $x, y$ for $S_1$, $z, w$ for $S_2$ and $u,v$ for $S_3$ as depicted.
		Observe that the coloring conditions on $S_2$ and $S_3$ immediately imply that $x = y$. 
		Set $z = \zeta^r$, $w = \zeta^s$, $u = \zeta^p$ and $v = \zeta^q$. 
On $S_1$, 
the coloring condition for the component labeled $x$ is  $x z^{-1} w u^{-1} v=x$, 
which is equivalent to  $\zeta^{s-r+q-p} = 1$,  and the $y$ component has the same condition.
		For each arbitrary choice of $r, s, p$ there exists a unique $q$ such that $\zeta^{s-r+q-p} = 1$. Therefore, also considering that $x = y$ can be chosen freely, it follows that there are $n^4$ colorings. 
	
	From the definition of the cocycle invariant we have that, setting $\bar \phi = (0,\phi, 0)$, $\Psi_{\bar \phi}(S)$ is a sum of six tensor products, the first two entries corresponding to the two boundary components of $S_1$ and, similarly, the other entries corresponding to the boundary components of $S_2$ and $S_3$. For each coloring, the Boltzmann weight on the boundary components of $S_2$ and $S_3$ are always trivial, since $\phi$ is non-degenerate and $x = y$. The Boltzmann weights on $S_1$ are identical for 
	both boundary components and are determined by 
	$z^{-1}w=\zeta^{s-r}$.
	The definition of $\psi=\sum_{i=0}^{n-1} i \psi_i$ is written multiplicatively
	 $\psi=\prod_{i=0}^{n-1} ( \psi_i )^i$, and $\psi_i(x, \zeta^r, \zeta^s)=g^{i}$ if and only if $i=s-r$ for any $x$,
	 where we have that $g$ is a multiplicative generator of $A$.
	Hence $\psi(x,z,w)=g^i$ for all $x$ and $z=\zeta^r$, $w=\zeta^s$ and $i=s-r$. 
	For each given $i$, there are  $n$ choices for $r$ and $s$ such that $s-r=i$, and for each choice of 
	$(r,s)$, there are  $n$ choices for $p$ and $q$ such that $\zeta^{s-r+q-p} = 1$,
	and independently there are $n$ choices for $x$. 
	Thus we obtain 
	$ \Psi_{\bar \phi}(S) = n^3 \sum_{i=0}^{n-1} ( g^i )^{\otimes 2} \otimes e^{\otimes 4} $.

}
\end{example}

From Theorem~\ref{thm:heap} and the preceding example, we have the following.

\begin{corollary}
We have $\bar H^{2,n}_{\rm RA} (\Z_m,  \Z_m) \neq 0$ for all $m>1$ and $n \geq 3$.
\end{corollary}

		We present below a realization result for surface ribbons with non-trivial invariant.

		\begin{figure}[htb]
			\begin{center}
				\includegraphics[width=1.5in]{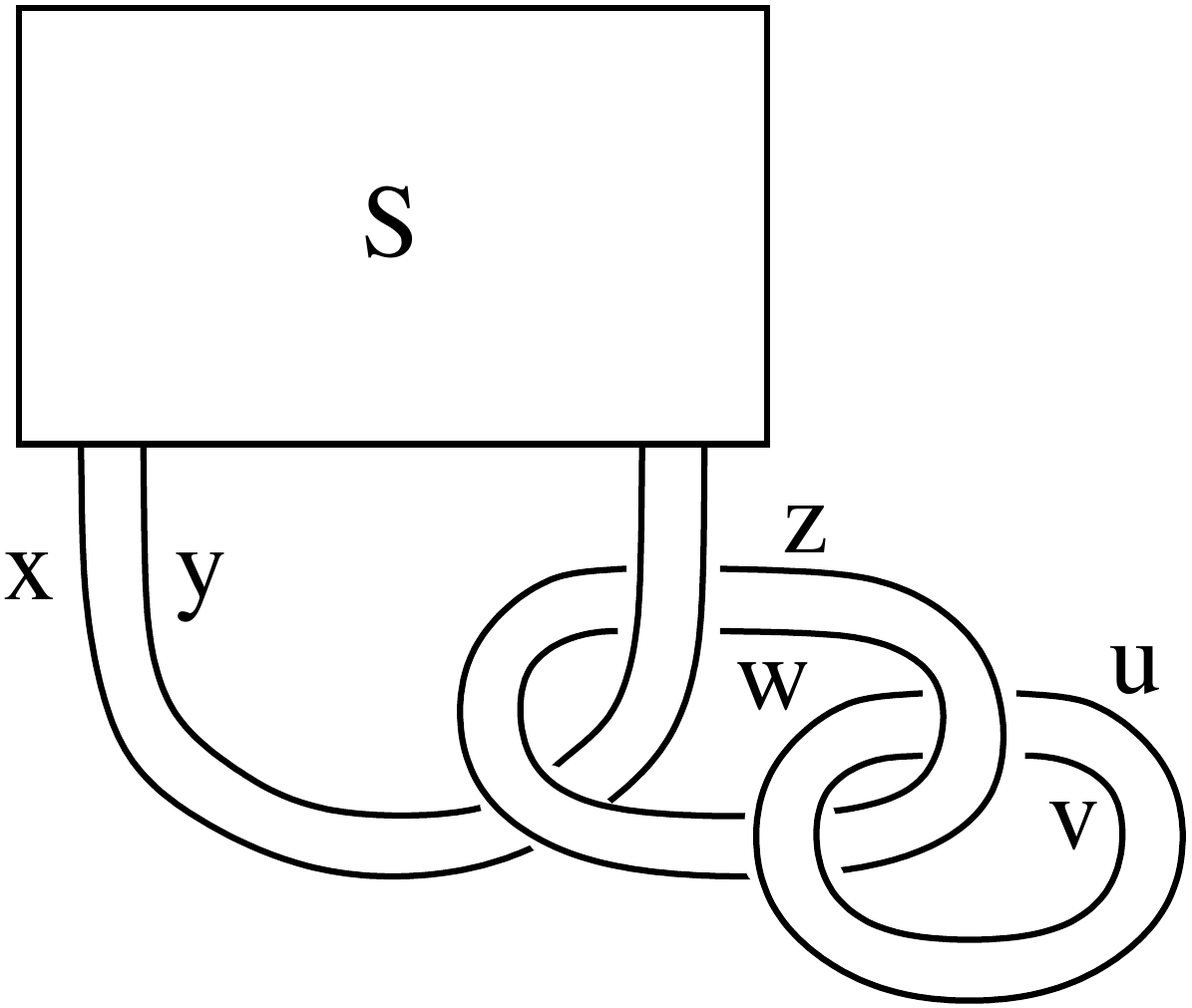}
			\end{center}
			\caption{Linking two rings} 
			\label{Hopfribbon}
		\end{figure}

\begin{proposition}
The following 
statements 
hold.
\begin{itemize}
	\item[(A)] 
	For every $n\geq 3$ and pairs $(g_i,b_i)$ with $g_i\geq 0$ and $b_i\geq 0$ for each $i = 1, \ldots , n$, there exists
	 an $n$-component non-split surface ribbon $S = S_1\cup \cdots \cup S_n$ with nontrivial cocycle invariant
	 for some group heap $X$ and coefficient group $A$, such that $S_i$ has genus $g_i$ and $b_i$ boundary components. 
	\item[(B)]
	Let $X$ be a heap, $A$ an abelian group and let $\bar \psi = (\psi_1, \ldots, \psi_n)$ be an $n$-tuple of mutually distributive (non-degenerate) cocycles. Suppose that $S = S_1 \cup \cdots \cup S_n$ is an $n$-component surface ribbon, where $S_i$ has genus $g_i$ and $b_i$ boundary components, such that $\Psi_{\bar \psi}(S)$ is nontrivial. Then there exists a  surface ribbon with nontrivial invariant with $n'$ connected components and $(g'_i,b'_i)$ for $i = 1, \ldots , n'$ such that:
	\begin{itemize}
		\item[(i)]
		$n'\geq n+2$, $g'_i \geq g_i$, $b'_i\geq b_i$ with $g'_{n+1}, g'_{n+2} \geq 0$ and $b'_{n+1}, b'_{n+2} \geq 0$;
		\item[(ii)]
		$n' \geq n$, $b_i' \geq  b_i$, and there exists $j$ such that $g_j'> g_j$, and for all $i \neq j$, $g_i' \geq g_i$. 
	\end{itemize}
\end{itemize}
\end{proposition}
\begin{proof}	
	Let $S = S_1 \cup \cdots \cup S_n$ be a non-split surface ribbon with $n$ connected components. Let $X$ be a heap and $A$ an abelian group in general. Suppose that $\Psi_{\bar \psi}(S)$ is nontrivial for some choice of $n$ decorating mutually distributive separable cocycles $\bar \psi = (\psi_1, \ldots , \psi_n)$. Since $S$ is non-split, none of the components is a disk,  hence $H_1(S_i) \neq 0$ so that there is  a nontrivial handle for all $i$. Let us assume further the condition that:
	
	(**) There is a 
	coloring ${\mathcal C}$ of $S$ such that it is monochromatic on $S_j$ and 
	 $\Psi_{\bar \psi}(S)$ is non-trivial. 
	
	Let $S' = S \cup R_1 \cup \cdots \cup R_m$ 
	be the surface ribbon obtained from $S$ by linking  
	annular 
	rings $R_1, \ldots, R_m$ on $S_j$
	(as  $S_2$ and $S_3$ linking $S_1$ in Figure~\ref{rings}).
	Let  ${\mathcal C}'$ be a coloring of $S'$ 
	 obtained by extending ${\mathcal C}$ with monochromatic colors on $R_1, \ldots, R_m$. 
	 Let  $\bar \psi'$ be obtained from $\bar \psi$ by appending $m$ zero cocycles decorating each ribbon ring $R_k$.
	 Then  we have that the invariant $\Psi_{\bar \psi'}(S')$ is nontrivial.
It follows that for each $k\geq n$ there is a non-split  surface ribbon with $k$ connected components whose corresponding cocycle invariant is nontrivial. 

Let $S= S_1\cup \cdots \cup S_n$ be a surface with nontrivial cocycle invariant satisfying the same conditions as above. Let us denote by $(g_i,b_i)$ the genus and the number of boundary components.
 By adding trivial bands to $S_i$ as in Figure~\ref{trivribbon} (A), we can increase the number of connected components $b_i$ of $S_i$ arbitrarily.
 Similarly, by attaching trivial torus band pairs as in Figure~\ref{trivribbon} (B), we can increase 
  the genus $g_i$ arbitrarily. 
  Applying Remark~\ref{pro:connectedcocy} we see that in both cases the invariant does not change and, therefore, it is nontrivial under either procedure.

 The surface of Example~\ref{ex:minimalnontrivial}  satisfies the required condition (**)  with $n = 3$ and $X=\Z_\ell=A$. Moreover, for $i = 1,2,3$ we have that $(g_i, b_i)=(0,2)$. From the preceding argument, then, it follows that there exists a non-split $k$-component surface ribbon with nontrivial cocycle invariant for every prescribed choice of 
 $n\geq 3$ 
  and for each choice of 
  $n$ 
  pairs $(g_i,b_i)$ with $g_i\geq 0$ and $b_i \geq 2$. This completes the proof of (A).
 
 Let us now prove (B). For the statement of (i), as before, we have that $S$ is nontrivial and 
 at least one connected component contributes a non-trivial value to
 $\Psi_{\bar \psi } (S)$.
 Suppose this is $S_j$. Since $S_j$ has a handle, 
 we add two ribbon surfaces as in Figure~\ref{Hopfribbon} and denote the resulting surface by $S'$. Let us decorate both added surfaces by some (non-degenerate by hypothesis) $\psi_i$. Denote $\bar \psi'$ the $(n+2)$-tuple obtained by adding two $\psi_i$ to $\bar \psi$. Since $\Psi_{\bar \psi}(S)$ is nontrivial, there is a coloring whose corresponding Boltzmann weight is nontrivial. Let us indicate by $x, y $ the corresponding colors of the arcs of  
 $S_j$. Then, taking in Figure~\ref{Hopfribbon} $z = w = u = 1$ and $v = x^{-1}y$, we obtain a coloring of $S'$, whose associated Boltzmann weight is nontrivial. Since this summand in the cocycle invariant is not canceled by other weights, it follows that $\Psi_{\bar \psi'}(S')$ is nontrivial. Observe that $S'$ has $n+2$ connected components, it has $(g_i,b_i)$ unchanged with respect to $S$ for all $i = 1, \ldots , n$, and it has $g_{n+1} = g_{n+2} = 0$, $b_{n+1} = b_{n+2} = 2$. Moreover, it satisfies the hypothesis $(**)$ 
 of part (A), with respect to the $(n+1)^{\rm st}$ component. We can apply part (A) to complete the proof of (i). To prove part (ii), let $S_j$ be a connected component of $S$ with a nontrivial handle. Observe that we can add to $S_j$ a torus band pair where the rightmost foot is monochromatic, augmenting the genus of $S$ by one unit. Now we can proceed as in (A) to complete.
\end{proof}

\begin{remark}\label{pro:connectedcocy}
{\rm 

Let us now consider 
the cocycle invariant under boundary connected sum.
Let $S_1 = S_1^1 \cup \cdots \cup S_1^n$ and $S_2 = S_2^1 \cup \cdots \cup S_2^m$,
be surface ribbons 
 with cocycle invariants $\Psi_{\bar \psi}(S_1)$ and $\Psi_{\bar \phi}(S_2)$, respectively, for some additive and reversible cocycle tuples of mutually distributive cocycles $\bar \psi, \bar \phi \in \bar Z^2_{\rm RA}(X,A)$ where $X$ is a heap, and $A$ is an abelian group. Suppose that $\psi_r = \phi_s$, and that the corresponding connected component  $S^r_1$ has $b$ boundary components, while $S^s_2$ has $b'$. When we perform the boundary connected sum of $S_1$ and $S_2$, with respect to $S_1^r$ and $S_2^s$, this construction is applied to two boundary components, which we assume being the $i^{\rm th}$ and $j^{\rm th}$ ones for $S^r_1$ and $S^s_2$, respectively. 
 In this situation, the connected components of $S_1\natural S_2$ are ordered as $(S_1^r\natural S_2^s, S_1^1, \ldots,\hat{S}_1^i, \ldots, S_1^n, S_2^1, \ldots , \hat{S}_2^j,\ldots S_2^m)$, where the symbol $\hat{\ }$ indicates omission of the component.
 The cocycles decorating the connected components are arranged in the tuple $\bar \psi\cdot_{rs} \bar \phi := (\psi_i, \psi_1, \ldots, \hat \psi_i 
 \ldots, \psi_n, \phi_1, \ldots, 
 \hat \phi_i 
 \ldots, \phi_m)$, where $\psi_i = \phi_j$ by assumption. 
Observe that the cocycle invariant associated to each connected component different from $S_1^r$ and $S_2^s$ remains unaltered from this procedure. 
Therefore, 
we can focus on the computation of the tensor component relative to $S_1^r\natural S_2^s$. For simplicity, we omit referring to the components that are unchanged, and for $x\in X$, we denote by $\mathcal C_1(x)$, the colorings of $S_1^r$, and similarly for $S_2^s$, where $x$ is the color assigned to the arc $b_1$ and $b_2$ of Figure~\ref{connect} (A). Certain choices of $x$ might not admit colorings $\mathcal C_i(x)$ for either value of $i$, depending on the surface ribbon that is being considered.
Let us define an element of the symmetric algebra 
$ S(\Z[A])$
as follows. 
If we have that 
$$\Psi^r_{\bar \psi} (S_1) = \sum_{\mathcal C_1} \otimes_t \prod_{k_1} B_{x_t(k_1)}(\mathcal C_1, \tau_t(k_1)) \quad {\rm  and } \quad
\Psi^s_{\bar \phi} (S_2) = \sum_{\mathcal C_2} \otimes_l \prod_{k_2} B_{x_l( k_2 )}(\mathcal C_2, \tau_l( k_2 ))$$
 we set 
\begin{eqnarray*}
\lefteqn{\Psi^r_{\bar \psi}(S_1) \cdot_{ij} \Psi^s_{\bar \phi}(S_2) }\\
&:=& \sum_x\sum_{\mathcal C_1(x), \mathcal C_2(x)} 
\left[ 
\prod_{k_1} B_{x_i(k_1)}(\mathcal C_1(x),\tau_i(k_1) ) 
\cdot \prod_\ell B_{x_j(k_2)}(\mathcal C_2(x), \tau_j(k_2))
\right]
\\
&&\otimes_{t\neq i} \prod_{k_1} B_{x_t(k_1)}(\mathcal C_1(x), \tau_t(k_1)) \otimes_{l\neq j} \prod_{k_2} B_{x_l(k_2)}(\mathcal C_2(x), \tau_l(k_2))),
\end{eqnarray*}
where $\tau$'s and $x_t(k_1)$,  
$x_l (k_2)$ 
follow the same conventions of Definition~\ref{def:cocyinvariant}. 
 In other words, $\Psi^r_{\bar \psi}(S_1) \cdot_{ij} \Psi^s_{\bar \phi}(S_2)$ is obtained from the cocycle invariants by juxtaposing the tensors corresponding to all the boundary components different from $i$ and $j$, while these latter entries are multiplied together to give a single entry in the tensor product. The sum runs over all colorings $\mathcal C_1$ of $S^r_1$ and $\mathcal C_2$ of $S^s_2$ assigning the same value $x$ to the arcs $b_1$ and $b_2$, and then $x$ is taken over all elements of $X$. By convention, if such a coloring does not exist, the corresponding summand is zero. 
 
 A coloring of a ribbon surface $S$ by a heap $X$ is the same as a heap morphism from the fundamental heap $h(S)$ to $X$. Applying Proposition~\ref{pro:connected} we see that  morphisms from the free product of the reduced heaps of $S_1$ and $S_2$ induce morphisms of $h(S_1\natural S_2)$, and therefore colorings of $S_1\natural S_2$ by $X$. 
 However, 
there may be 
colorings that do not arise in this way. 
 They are colorings such that the arcs corresponding to $b_1$ and $b_2$ have distinct colorings, that correspond to $x_1'$ and $x_2'$ in the proof of Proposition~\ref{pro:connected},
 that do not come from colorings of $S_1$ and $S_2$, as discussed in the proof below in more details. 
 We say that these are the {\it residual colorings} of $S_1\natural S_2$. They are characterized by the fact that they do not factor through the free product $\hat h(S_1)* \hat h(S_2)$. 
 We examine the (hypothetical) invariant factors if there is a non-empty residual colorings.

	Let $S_1$ and $S_2$ be as above, and let $\Psi^r_{\bar \psi}(S_1)$ and $\Psi^s_{\bar \phi}(S_2)$ be their cocycle invariants componentes relative to $S_1^r$ and $S_2^s$, respectively. Suppose that $S_1\natural S_2$ is the boundary connected sum along the $i^{\rm th}$ boundary component of $S^r_1$ and the $j^{\rm th}$ boundary component of $S^s_2$. Then we have
	$$
	\Psi^k_{\bar{\psi}\cdot_{rs}\bar{\phi}}(S_1\natural S_2) =
	\begin{cases}
	 \Psi^r_{\bar \psi}(S_1)\cdot_{ij} \Psi^s_{\bar \phi}(S_2) + \sum_{\mathcal C^*} \otimes_{r = 1}^{b+b'-1} \prod_d B_{x_r(d)}(\mathcal C^*, \tau_r(d))\ \ {\rm if}\ \ k = 1\\
	 \Psi^{k-1}_{\bar \psi}(S_1) \ \ {\rm if}\ \ 1\neq k \leq i-1\\
	 \Psi^{k}_{\bar \psi}(S_1) \ \ {\rm if}\ \ i\leq  k \leq n\\
	 \Psi^{k}_{\bar \phi}(S_2) \ \ {\rm if}\ \ n+1\leq  k \leq n+j-1\\
	 \Psi^{k-1}_{\bar \phi}(S_2) \ \ {\rm if}\ \ n+j\leq  k \leq n+m
	\end{cases}
	$$
	where $\mathcal C^*$ denotes the residual colorings of $S_1\natural S_2$, and the ordering of the boundary components is as described above. 

	Let us now prove this claim. As observed above, the invariants relative to connected components different from $S_1^r$ and $S_2^s$ remain unchanged, with the only difference being the numbering of the connected component they refer to. This is the same as in the 
	formula stated above. 
	Let us consider $S^r_1\natural S^s_2$ as in Figure~\ref{connect} (B), where $b_1$ and $b_2$ of the figure correspond to the $i^{\rm th}$ boundary component of $S_1$ and the $j^{\rm th}$ boundary component of $S_2$, respectively. 
	We call these components $b^1_i$ and $b^2_j$. 
	Colorings of $S_1\natural S_2$
	 arise from the fundamental heap as follows. For each choice of $x\in X$, one colors the upper arc of the connecting band of Figure~\ref{connect} (B) with $x$ and, proceeding counterclockwise, goes along $b_1$. The arc outgoing from $S^r_1$, 
	 which is the lower arc of the band in Figure~\ref{connect} (B), has now picked a color 
	 $x\Phi_1$ 
	 which will enter $S^s_2$. 
	 After proceeding along $b^2_j$ in $S_2$ 
	 we have a color $x\Phi_1\Phi_2$ which forces a condition $\Phi_1\Phi_2 = 1$. This is the extra relation of Proposition~\ref{pro:connected}. Denoting $\mathcal C_1(x)$ and $\mathcal C_2(x)$ colorings of $S_1$ and $S_2$ 
	 such that $x$ is the color assigned to the arcs of $b^1_i$ and $b^2_j$, we see that they give rise to a coloring of $S_1\natural S_2$ 
	 since $\mathcal C_1(x)$ and $\mathcal C_2(x)$ would both map $\Phi_1$ and $\Phi_2$ to $1$. These colorings factor through the free product $\hat h(S_1)*\hat h(S_2)$ since $\Phi_i$ are relators of $\hat h(S_i)$. The contribution of these colorings to the cocycle invariant is obtained by multiplying the tensor entries corresponding to the boundary components $b^1_i$ and $b^2_j$. In fact, since they are now connected through a band, they constitute a single boundary component and ordering the arcs counterclockwise, as in the preceding considerations regarding the colorings, one obtains a product of Boltzmann weights that has all the crossings of $S_1$ first, and all the crossings of $S_2$ afterwards. The other boundary components are not modified by this procedure. We obtain, letting $x$ vary in $X$, the term $\Psi^r_{\bar \psi}(S_1)\cdot_{ij}\Psi^s_{\bar \phi}(S_2)$ 
	 in the statement. The contribution given by colorings that do not map $\Phi_i$ to $1$, but that map $\Phi_1\Phi_2$ to $1$, is due to the residual colorings by definition. This is the second summand appearing in the formula above. 
%

  We observe that the residual colorings of a boundary connected sum of surfaces are automatically trivial if $\Gamma$ of Proposition~\ref{pro:connected} is an isomorphism. Therefore, if $\Gamma$ is monic, it would follow that the formula in Remark~\ref{pro:connectedcocy} (for the $r$ component) would reduce to the product of cocycle invariants $\cdot_{ij}$. It is not clear, at this point, whether $\Gamma$ is always an isomorphism or not. 
}
\end{remark}

Although it is beyond the scope of this paper, it is desirable to investigate 
further properties of the fundamental heap and cocycle invariants of surface ribbons, 
and relations to other invariants of surfaces.

   \bigskip        	

\noindent
{\bf Acknowledgements.} 
MS was supported in part by NSF DMS-1800443.
EZ was 
supported by the Estonian Research Council via the Mobilitas Pluss scheme, grant MOBJD679.

\end{document}